\documentclass{amsart}

\usepackage[latin1]{inputenc}
\usepackage{amsfonts}
\usepackage{amsmath}
\usepackage{mathtools}
\usepackage{amsthm}
\usepackage{amssymb,color}
\usepackage{latexsym}
\usepackage{enumerate}
\usepackage{tikz-cd} 
\usepackage{enumitem}
\usepackage[colorlinks]{hyperref}
\usepackage{rotating}

\newtheorem{theorem}{Theorem}[section]
\newtheorem{lemma}[theorem]{Lemma}
\newtheorem{proposition}[theorem]{Proposition}
\newtheorem{corollary}[theorem]{Corollary}

\numberwithin{equation}{section}

\theoremstyle{definition}
\newtheorem{definition}[theorem]{Definition}

\newtheorem{remark}[theorem]{Remark}

\theoremstyle{remark}

\newtheorem*{notation}{Notation}

\newcommand{\surj}{\twoheadrightarrow}
\newcommand{\hookr}{\hookrightarrow}
\newcommand{\A}{\mathcal{A}}
\newcommand{\C}{\mathcal{C}}
\newcommand{\D}{\mathcal{D}}
\newcommand{\U}{\mathcal{U}}
\newcommand{\B}{\mathcal{B}}
\newcommand{\F}{\mathcal{F}}
\newcommand{\E}{\mathcal{E}}

\newcommand{\Q}{\mathcal{Q}}
\newcommand{\J}{\mathcal{J}}
\newcommand{\KK}{\mathfrak{K}}
\newcommand{\LL}{\mathfrak{L}}
\newcommand{\G}{\mathcal{G}}

\newcommand{\la}{\langle}
\newcommand{\ra}{\rangle}
\newcommand{\Mult}{\operatorname{Mult}}
\newcommand{\Ad}{\operatorname{Ad}}
\newcommand{\id}{\operatorname{id}}

\begin{document}

\title{Universal AF-algebras}

\author{Saeed Ghasemi}
\address{Institute of Mathematics, Czech Academy of Sciences, Czechia}

\author{Wies\l aw Kubi\'s}
\address{Institute of Mathematics, Czech Academy of Sciences, Czechia \\
Institute of Mathematics, Cardinal Stefan Wyszy\'nski University in Warsaw, Poland}

\thanks{Research of the first author was supported by the GA\v CR project 19-05271Y and RVO: 67985840. Research of the second author was supported by the GA\v CR project  20-31529X and RVO: 67985840.}

\begin{abstract}
We study the approximately finite-dimensional (AF) $C^*$-algebras that appear as inductive limits of sequences of finite-dimensional $C^*$-algebras and left-invertible embeddings. 
 We show that there is such a separable AF-algebra $\A_\mathfrak{F}$ which is a split-extension of any  finite-dimensional $C^*$-algebra and has the property that any separable AF-algebra is isomorphic to a quotient of $\A_\mathfrak{F}$. Equivalently, by Elliott's classification of separable AF-algebras, there are surjectively universal countable scaled (or with order-unit) dimension groups.
 This universality is a consequence of our result stating that  $\A_\mathfrak{F}$ is the Fra\"\i ss\'e limit of the category of all finite-dimensional $C^*$-algebras and left-invertible embeddings.  
 
 With the help of Fra\"\i ss\'e theory we describe the Bratteli diagram of $\A_\mathfrak{F}$ and provide conditions characterizing it up to isomorphisms.  $\A_\mathfrak{F}$ belongs to a class of separable AF-algebras which are all Fra\"\i ss\'e limits of suitable categories of finite-dimensional $C^*$-algebras, and resemble $C(2^\mathbb N)$ in many senses. For instance, they have no minimal projections, tensorially absorb $C(2^\mathbb N)$ (i.e. they are $C(2^\mathbb N)$-stable) and satisfy similar homogeneity and universality properties as the Cantor set.

\ 



\end{abstract}

\maketitle

\section{Introduction}
Operator algebraists often refer to (for good reasons, of course) the UHF-algebras such as CAR-algebra as the noncommutative analogues of the Cantor set $2^\mathbb N$, or more precisely the commutative $C^*$-algebra $C(2^\mathbb N)$. We introduce a different class of separable AF-algebras, we call them ``AF-algebras with Cantor property"  (Definition \ref{Bratteli}), which in some contexts are more suitable noncommutative analogues of $C(2^\mathbb N)$. One of the main features of AF-algebras with  Cantor property is that  they are direct limits of sequences of finite-dimensional $C^*$-algebras where the connecting maps are left-invertible homomorphisms. This property, for example, guarantees that if the algebra is infinite-dimensional, it has plenty of nontrivial ideals and quotients, while UHF-algebras are simple. The Cantor set is a ``special and unique" space in the category of all compact (zero-dimensional) metrizable spaces in the sense that it bears some universality and homogeneity properties; it maps onto any compact (zero-dimensional) metrizable space and it has the homogeneity property that any homeomorphism between finite quotients lifts to a homeomorphism of the Cantor set (see \cite{Kubis-FS}). Moreover, Cantor set is the unique compact zero-dimensional metrizable space with the property that (stated algebraically): for every $m,n\in \mathbb N$ and unital embeddings $\phi: \mathbb C^n \to \mathbb C^m$ and $\alpha: \mathbb C^n \to C(2^\mathbb N)$ there is an embedding $\beta:\mathbb C^m \to C(2^\mathbb N)$ such that the diagram 
\begin{equation*}
\begin{tikzcd}
 &  C(2^\mathbb N)   \\ 
\mathbb C^n \arrow[ur, hook, "\alpha"] \arrow[r, hook, "\phi"] & \mathbb C^m \arrow[u, hook, dashed, "\beta"]
\end{tikzcd}
\end{equation*}
commutes. Note that the map $\phi$ in the above must be left-invertible and if $\alpha$ is left-invertible then $\beta$ can be chosen to be left-invertible. Recall that a homomorphism $\phi: \B\to \A$ is left-invertible if there is a homomorphism $\pi: \A \to \B$ such that $\pi \circ \phi = \id_\B$.
The AF-algebras with Cantor property satisfy similar universality and homogeneity properties in their corresponding categories of finite-dimensional $C^*$-algebras and left-invertible homomorphisms.   Although, in general AF-algebras with  Cantor property are not assumed to be unital, when restricted to the categories with unital maps, one can obtain the unital AF-algebras with same properties subject to the condition that maps are unital.
For instance, the ``truly" noncommutative AF-algebra with  Cantor property $\A_\mathfrak F$, that was mentioned in the abstract,  is the unique (nonunital) AF-algebra which is the limit of a sequence of finite-dimensional $C^*$-algebras and left-invertible homomorphisms (necessarily embeddings), with the property that for every finite-dimensional $C^*$-algebras $\D, \E$ and (not necessarily unital) left-invertible embeddings  $\phi: \D \to \E$ and $\alpha: \D \to \A_\mathfrak F$ there is a left-invertible embedding $\beta:\E \to \A_\mathfrak F$ such that the diagram 
\begin{equation*}
\begin{tikzcd}
 &   \A_\mathfrak F  \\ 
\D \arrow[ur, hook, "\alpha"] \arrow[r, hook, "\phi"] & \E \arrow[u, hook, dashed, "\beta"]
\end{tikzcd}
\end{equation*}
commutes (Theorem \ref{A_F-char}).
One of our main results (Theorem \ref{universal AF-algebra}) states that $\A_\mathfrak F$ maps surjectively onto any separable AF-algebra. However, this universality property is not unique to $\A_\mathfrak F$ (Remark \ref{not unique}). 

The properties of the Cantor set that are mentioned above can be viewed as consequences of the fact that it is the ``Fra\"\i ss\'e  limit" of the class of all nonempty finite spaces  and surjective maps (as well as the class of all  nonempty compact metric spaces and continuous surjections); see \cite{Kubis-ME}. 
The theory of Fra\"\i ss\'e limits was introduced by R. Fra\"\i ss\'e \cite{Fraisse} in 1954 as a model-theoretic approach to the back-and-forth argument.  Roughly speaking, Fra\"\i ss\'e  theory establishes a correspondence between classes of finite (or finitely generated) models of a first-order language with certain properties (the joint-embedding property, the amalgamation property and having countably many isomorphism types), known as \emph{Fra\"\i ss\'e classes}, and the unique (ultra-)homogeneous and universal countable structure, known as the \emph{Fra\"\i ss\'e limit}, which can be represented as the union of a chain of models from the class. 
Fra\"\i ss\'e theory has been recently extended way beyond the countable first-order structures,
in particular, covering some topological spaces, Banach spaces and, even more recently, some $C^*$-algebras. Usually in these extensions the classical Fra\"\i ss\'e  theory is replaced by its ``approximate" version.  Approximate Fra\"\i ss\'e theory was developed by Ben Yaacov \cite{Ben Yaacov} in continuous model theory (an earlier approach was developed in \cite{Schoretsanitis}) and independently, in the framework of metric-enriched categories, by the second author~\cite{Kubis-ME}.
The Urysohn metric space, the separable infinite-dimensional Hilbert space \cite{Ben Yaacov}, and  the Gurari\u{\i} space \cite{Kubis-Solecki} are some of the other well known examples of Fra\"\i ss\'e limits of metric structures (see also \cite{Martino} for more on Fra\"\i ss\'e  limits in functional analysis).  

Fra\"\i ss\'e  limits of $C^*$-algebras are studied in  \cite{Eagle} and \cite{Masumoto}, where it has been shown that the Jiang-Su algebra, all UHF algebras, and the hyperfinite II$_1$-factor are  Fra\"\i ss\'e  limits of suitable classes of finitely generated $C^*$-algebras  with distinguished traces. Here we investigate the separable AF-algebras that arise as  limits of Fra\"\i ss\'e classes of finite-dimensional $C^*$-algebras.  Apart from $C(2^\mathbb N)$, which is the Fra\"\i ss\'e  limit of the class of all commutative finite-dimensional $C^*$-algebras and unital (automatically left-invertible) embeddings, all UHF-algebras  \cite[Theorem 3.4]{Eagle} and a class of simple monotracial AF-algebras described in \cite[Theorem 3.9]{Eagle}, are Fra\"\i ss\'e  limits of classes of finite-dimensional $C^*$-algebras.   
It is also worth noticing that the $C^*$-algebra $\mathcal K(H)$ of all compact operators on a separable Hilbert space $H$ and the universal UHF-algebra $\Q$ (see Section \ref{universal uhf}) are both Fra\"\i ss\'e limits of, respectively, the category of all  matrix algebras and (not necessarily unital) embeddings and the category of all matrix algebras and unital embeddings. 

In general, however, obstacles arising from the existence of traces prevent many classes of finite-dimensional $C^*$-algebras from having the amalgamation property (\cite[Proposition 3.3]{Eagle}), therefore making it difficult to realize AF-algebras as Fra\"\i ss\'e limits of such classes. The AF-algebra $C(2^\mathbb N)$ is neither a UHF-algebra nor it is among AF-algebras considered in \cite[Theorem 3.9]{Eagle}. Therefore, 
it is natural to ask whether
$C(2^\mathbb N)$ belongs to any larger nontrivial class of AF-algebras whose elements are Fra\"\i ss\'e limits of some class of finite-dimensional $C^*$-algebras. This was our initial motivation behind introducing the class of separable AF-algebras with Cantor property (Definition \ref{Bratteli}).
This class properly contains the AF-algebras of the form $M_n \otimes C(2^\mathbb N)$, for any matrix algebra $M_n$.

If $\A$ and $\B$ are $C^*$-algebras,  $\phi: \B \hookr \A$ is a left-invertible embedding and $\pi:\A \surj \B$ is a left inverse of $\phi$, then we have the short exact sequence 
\begin{equation*}
\begin{tikzcd}[ ampersand replacement=\&]
0 \arrow[r, hook] \& \ker(\pi) \arrow[r, hook,  "\iota"] \& \A \arrow[r, two heads, shift left, "\pi"]  \& \B \arrow[l, hook', shift left,  "\phi"] \arrow[r, hook] \& 0.
\end{tikzcd}
\end{equation*}
Therefore $\A$ is a ``split-extension" of $\B$. In this case we say $\B$ is a ``retract" of $\A$. It would be more convenient for us to say `` $\B$ is a retract of $\A$" rather than the more familiar phrase (for $C^*$-algebraists) ``$\A$ is a split-extension of $\B$".   In Section \ref{LI-section} we consider direct sequences of finite-dimensional $C^*$-algebras
$$\A_1 \xrightarrow{\phi_1^2}\A_2 \xrightarrow{\phi_2^3}\A_3 \xrightarrow{\phi_3^4} \dots $$
where each $\phi_n^{n+1}$ is a left-invertible embedding. The  AF-algebra $\A$ that arises as the limit of this sequence has the property that every matrix algebra $M_k$ appearing as a direct-sum component (an ideal) of some $\A_n$ is a retract of $\A$ (equivalently, $\A$ is a split-extension of each $\A_n$ and such $M_k$). Moreover, every retract of $\A$ which is a matrix algebra, appears as a direct-sum component of some $\A_n$ (Lemma \ref{sub-retract}).  The AF-algebras with Cantor property are defined and studied in Section \ref{CP-section}. They are characterized by the set of their matrix algebra retracts. That is, two AF-algebras with Cantor property are isomorphic if and only if they have exactly the same matrix algebras as their retracts (Corollary \ref{retract-coro}), i.e., they are split-extensions of the same class of matrix algebras.

We will use the Fra\"\i ss\'e-theoretic framework of (\emph{metric-enriched}) categories described in \cite{Kubis-ME}, rather than the (metric) model-theoretic approach to the Fra\"\i ss\'e theory. A brief introduction to Fra\"\i ss\'e  categories is provided in Section \ref{Flim-intro-section}.
We show that (Theorem \ref{closed-Fraisse}) any category of finite-dimensional $C^*$-algebras and (not necessarily unital) left-invertible embeddings, which is closed under taking direct sums and ideals of its objects (we call these categories  $\oplus$-stable) is a Fra\"\i ss\'e category.  Moreover,  Fra\"\i ss\'e limits of these categories have the Cantor property (Lemma \ref{A_K-CS}) and in fact any AF-algebra $\A$ with Cantor property can be realized as Fra\"\i ss\'e limit of such a category, where the objects of this category are precisely the finite-dimensional retracts of $\A$ (see Definition \ref{retract-def} and Theorem \ref{Cantor-closed-Fraisse}).  

 In particular, the category $\mathfrak F$ of \emph{all} finite-dimensional $C^*$-algebras and left-invertible embeddings  is a Fra\"\i ss\'e category (Section \ref{universal-section}). A priori, the Fra\"\i ss\'e limit $\mathcal A_{\mathfrak F}$ of this category is a separable AF-algebra with the universality property that any separable AF-algebra $\mathcal A$ which is the limit of a sequence of finite-dimensional $C^*$-algebras with left-invertible embeddings as connecting maps, can be embedded into $\A_\mathfrak F$ via a left-invertible embedding, i.e., $\A_\mathfrak F$ is a split-extension of $\A$. In particular, there is a surjective homomorphism $\theta :\A_\mathfrak F \surj \A$. Also any separable AF-algebra is isomorphic to a quotient (by an essential ideal) of an AF-algebra which is the limit  of a sequence of finite-dimensional $C^*$-algebras with left-invertible embeddings (Proposition \ref{ess-quotient}). Combining the two quotient maps, we have the following result, which is later restated as Theorem \ref{universal AF-algebra}.

\begin{theorem}\label{main theorem}
 The  category of all finite-dimensional $C^*$-algebras and left-invertible embeddings is a  Fra\"\i ss\'e category. Its Fra\"\i ss\'e  limit $\mathcal A_{\mathfrak F}$ is a separable AF-algebra such that
 \begin{itemize}
     \item $\mathcal A_{\mathfrak F}$ is a split-extension of any AF-algebra which is the limit of a sequence of finite-dimensional $C^*$-algebras and left-invertible connecting maps.
     \item  there is a surjective homomorphism from $\A_{\mathfrak F}$ onto any separable AF-algebra.
 \end{itemize}
\end{theorem}
The Bratteli diagram of $\A_\mathfrak F$ is described in Proposition \ref{Bratteli diagram-A_F}, using the fact that it has the Cantor property. It is the unique AF-algebra with Cantor property which is a split-extension of every finite-dimensional $C^*$-algebra. The unital versions of these results are given in Section \ref{unital-sec} (with a bit of extra work, since unlike $\mathfrak F$, the category of all finite-dimensional $C^*$-algebras and \emph{unital} left-invertible maps is not a  Fra\"\i ss\'e category, namely, it lacks the joint embedding property).

Separable AF-algebras are famously characterized \cite{Elliott} by their $K_0$-invariants which are scaled countable dimension groups (with order-unit, in the unital case). By applying the $K_0$-functor to Theorem \ref{main theorem} we have the following result.
\begin{corollary}
There is a scaled countable dimension group (with order-unit) which maps onto any  scaled  countable dimension group (with order-unit).
\end{corollary}
The corresponding characterizations of these dimension groups are mentioned in Section \ref{K0-section}.

Finally, this paper could have been written entirely in the language of partially ordered abelian groups, where the categories of ``simplicial groups" and left-invertible positive embeddings replace our categories. However, we do not see any clear advantage in doing so.

\section{Preliminaries}\label{pre-sec}
Recall that an approximately finite-dimensional (AF) algebra is a $C^*$-algebra which is an inductive limit of a sequence of finite-dimensional $C^*$-algebras. 
 We review a few basic facts about separable AF-algebras. The background needed regarding AF-algebras is quite elementary and \cite{Davidson} is more than sufficient. The AF-algebras that are considered here are always separable and therefore by ``AF-algebra" we always mean ``separable AF-algebra".
 AF-algebras can be characterized up to isomorphisms by their Bratteli diagrams \cite{Bratteli}. However, there is no efficient way (at least visually) to decide whether two Bratteli diagrams are isomorphic, i.e., they correspond to isomorphic AF-algebras. A much better characterization of AF-algebras uses $K$-theory. To each $C^*$-algebra the $K_0$-functor assigns a partially ordered abelian group (its $K_0$-group) which turns out to be a complete invariant  for AF-algebras \cite{Elliott}. Moreover, there is a complete description of all possible $K_0$-groups of AF-algebras. Namely, a partially ordered abelian group is isomorphic to the $K_0$-group of an AF-algebra if and only if it is a countable dimension group.  
 
 We mostly use the notation from \cite{Davidson} with minor adjustments. 
  Let $M_k$ denote the $C^*$-algebra of all $k\times k$ matrices over $\mathbb C$.
 Suppose $\A = \varinjlim (\A_n, \phi_n^m)$ is an AF-algebra with Bratteli diagram $\mathfrak D$ such that each $\A_n \cong \A_{n,1} \oplus \dots\oplus \A_{n,\ell}$, is a finite-dimensional $C^*$-algebra and each $\A_{n,s}$ is a full matrix algebra. The node of $\mathfrak D$ corresponding to $\A_{n,s}$ is  ``officially" denoted by $(n,s)$, while intrinsically it carries over a natural number $\dim(n,s)$, which represents the dimension of the matrix algebra $\A_{n,s}$, i.e. $\A_{n,s} \cong M_{\dim(n,s)}$.
For  $(n,s) , (m,t) \in \mathfrak D$ we write $(n,s) \to (m,t)$ if $(n,t)$ is connected $(m,t)$ by at least one path  in $\mathfrak D$,  i.e. if $\phi_n^m$ sends $\A_{n,s}$ faithfully into $\A_{m,t}$.

 The ideals of AF-algebras are also AF-algebras and they can be recognized from the Bratteli diagram of the algebra. Namely, the  Bratteli diagrams of ideals correspond to \emph{directed} and \emph{hereditary} subsets of the Bratteli diagram of the algebra (see \cite[Theorem III.4.2]{Davidson}).  Recall that an essential ideal $\J$ of $\A$ is an ideal which has nonzero intersections with every nonzero  ideal of $\A$. Suppose $\mathfrak D$ is the Bratteli diagram for an AF-algebra $\A$ and $\J$ is an ideal of $\A$ whose Bratteli diagram corresponds to $\mathfrak J \subseteq \mathfrak D$. Then $\J$ is essential if and only if for every $(n,s)\in \mathfrak D$ there is $(m,t)\in \mathfrak J$ such that $(n,s) \to (m,t)$.

If $\D =\D_1 \oplus \dots \oplus \D_l $ and $  \E = \E_1 \oplus \dots \oplus \E_k$ are finite-dimensional $C^*$-algebras where $\D_i$ and $\E_j$ are matrix algebras and $\phi: \D \to \E$ is a homomorphism, we denote the ``multiplicity of $\D_i$ in $\E_j$ along $\phi$" by $\Mult_\phi(\D_i, \E_j) $. Also let $\Mult_\phi(\D, \E_j) $ denote the tuple 
$$(\Mult_\phi(\D_1, \E_j), \dots , \Mult_\phi(\D_l, \E_j))\in \mathbb N^l.$$ 
Suppose $\pi_j: \E \to \E_j$ is the canonical projection.
If $\Mult_\phi(\D, \E_j) = (x_1, \dots, x_l) $ then the group homomorphism $K_0(\pi_j \circ \phi): \mathbb Z^l \to \mathbb Z$  sends $(y_1, \dots, y_l)$ to $\sum_{i\leq l} x_i y_i$. 
Therefore if $\phi, \psi : \D \to \E$ are homomorphisms, we have $K_0(\phi) = K_0(\psi)$ if and only if $\Mult_\phi(\D, \E_j)= \Mult_\psi(\D, \E_j)$ for every $j\leq k$.

The following  well known facts about AF-algebras will be used several times throughout the article. We denote the unitization of $\A$ by $\widetilde \A$ and if $u$ is a unitary in $\widetilde\A$, then $\Ad_u$ denotes the inner automorphisms of $\A$ given by $a \to u^* a u$.

\begin{lemma}{\cite[Lemma III.3.2]{Davidson}}\label{twisting} Suppose $\epsilon >0$ and $\{\A_n\}$ is an increasing sequence of finite-dimensional $C^*$-algebras such that $\A = \overline{\bigcup \A_n}$. If $\F$ is a finite-dimensional subalgebra of $\A$, then there are $m\in \mathbb N$ and a unitary $u$ in $\widetilde \A$ such that $u^* \F u \subseteq \A_m $ and $\|1-u\| < \epsilon$.
\end{lemma}

\begin{lemma}\label{fd-af}
Suppose $\D$ is a finite-dimensional $C^*$-algebra, $\A$ is a separable AF-algebra and $\phi,\psi: \D \to \A$ are homomorphisms such that $\|\phi - \psi\|< 1$. Then there is a unitary $u\in \widetilde \A$ such that $\Ad_u \circ \psi = \phi$.   
\end{lemma}
\begin{proof}
We have $K_0(\phi) = K_0(\psi)$, since otherwise for some  nonzero projection $p$ in $\D$ the dimensions of the projections $\phi(p)$ and $\psi(p)$ differ and hence $\|\psi - \phi \| \geq 1$. Therefore there is a unitary $u$ in $\widetilde \A$ such that $\Ad_u \circ \psi  = \phi$, by \cite[Lemma 7.3.2]{K-theory Rordam}.
\end{proof}

\begin{lemma}\label{matrix-absorbing}
Suppose $\D = \D_1 \oplus \dots \oplus \D_l$ is a finite-dimensional $C^*$-algebra, where each  $\D_i$ is a  matrix algebra. Assume $\gamma: \D \hookr M_k$ and $\phi: \D \hookr M_\ell$ are embeddings. The following are equivalent. 
\begin{enumerate}
 \item There is an embedding $\delta: M_k \hookr M_\ell$ such that $\delta \circ \gamma = \phi$.
\item There is an embedding $\delta: M_k \hookr M_\ell$ such that $\|\delta \circ \gamma  - \phi \| <1$.
 \item  There is a natural number $c\geq 1$ such that $\ell \geq c k$ and  $\Mult_\phi(\D, M_\ell)  = c \Mult_\gamma(\D, M_k) $.
 \end{enumerate}
\end{lemma}
\begin{proof}
(1) trivially implies (2). To see (2)$\Rightarrow$(3), note that we have $$\Mult_\phi(\D_i, M_\ell)  = \Mult_\delta(M_k, M_\ell) \Mult_\gamma(\D_i, M_k), $$
for every $i\leq l$,  since otherwise  $\|\delta \circ \gamma  - \phi \| \geq 1$. Let  $c = \Mult_\delta(M_k, M_\ell)$.  To see (3)$\Rightarrow$(1), let $\delta': M_k \to M_l$ be the embedding which sends an element of $M_k$ to $c$ many identical copies of it along the diagonal of $M_\ell$. Then we have $K_0(\phi) = K_0(\delta' \circ \gamma)$, by the assumption of (3). Therefore there is a unitary $u$ in $M_\ell$ such that $\Ad_u \circ\delta' \circ \gamma = \phi$. Let $\delta= \Ad_u \circ\delta' $.   
\end{proof}

\section{AF-algebras with  left-invertible connecting maps}\label{LI-section}

Suppose $\A, \B$ are $C^*$-algebras.
A homomorphism $\phi: \B \to \A$ is  \emph{left-invertible} if  there is a (necessarily surjective) homomorphism $\pi: \A \surj \B$ such that $\pi \circ \phi = \id_{\B}$. Clearly a left-invertible homomorphism is necessarily an embedding.

\begin{definition}\label{retract-def}  We say $\B$ is a \emph{retract} of $\A$ if there is a left-invertible embedding from $\B$ into $\A$. 
We say a subalgebra $\B$ of $\A$ is an \emph{inner} retract if and only if there is a homomorphism $\theta: \A \surj \B$ such that $\theta|_{\B} = \id_\B$.
\end{definition}

 The image of a left-invertible embedding $\phi: \B \hookrightarrow \A$ is an inner retract of $\A$. Note that $\B$ is a retract of $\A$ if and only if $\A$ is a split-extension of $\B$.
The next proposition contains some elementary facts about retracts of finite-dimensional $C^*$-algebras and left-invertible maps between them. They follow from elementary facts about finite-dimensional $C^*$-algebras, e.g., matrix algebras are simple.
 \begin{proposition} \label{fact1}  A $C^*$-algebra
$\D$ is a retract of a finite-dimensional $C^*$-algebra $\E$ if and only if 
$\E \cong \D \oplus \F$, for some finite-dimensional $C^*$-algebra $\F$. In other words,  
$\D$ is a retract of $\E$ if and only if $\D$ is isomorphic to an ideal of $\E$.

 Suppose $\phi: \D \hookr \E$ is a (unital) left-invertible embedding and $\pi: \E \surj \D$ is a left inverse of $\phi$. Then
 $\E$ can be written as $\E_0 \oplus \E_1$ and there are $\phi_0, \phi_1$  such that $\phi_0 : \D \to \E_0$ is an isomorphism, $\phi_1: \D \rightarrow \E_1$ is a (unital) homomorphism and 
 \begin{itemize}
     \item $\phi(d)= (\phi_0(d), \phi_1(d))$, for every $d\in \D$,
     \item $\pi(e_0, e_1) = \phi^{ -1}_0(e_0)$, for every $(e_0, e_1) \in \E_0 \oplus \E_1$.
 \end{itemize}
 \end{proposition}

Suppose $(\A_n, \phi_n^m)$ is a  sequence where each connecting map $\phi_n^m: \A_n \hookr \A_m$ is left-invertible. Let $\pi_n^{n+1}: \A_{n+1} \surj \A_n$ be a left inverse of $\phi_n^{n+1}$, for each $n$. For $m> n$ define $\pi_n^m: \A_m \surj \A_n$ by  $\pi_n^m = \pi^{n+1}_n \circ \dots \circ \pi_{m-1}^m$. Then $\pi_n^m$ is a left inverse of $\phi_n^m$ which satisfies $\pi_n^m \circ \pi^k_m = \pi_n^k$, for every $n\leq m\leq k$.
\begin{definition} \label{left-invertible-AF}
 We say $(\A_n, \phi_n^m)$  is a \emph{left-invertible sequence} if  each $\phi_n^m$ is  left-invertible and $\phi_n^n = \id_{\A_n}$. We call $(\pi_n^m)$  a \emph{compatible} left inverse of the left-invertible sequence $(\A_n, \phi_n^m)$ if $\pi_n^m: \A_m \surj \A_n$ are surjective homomorphisms such that $\pi_n^m \circ \pi^k_m = \pi_n^k$ and $\pi_n^m \circ \phi_n^m = \id_{\A_n}$, for every $n\leq m\leq k$. 
 \end{definition}

The following simple lemma is true for arbitrary categories, see~\cite[Lemma 6.2]{Kubis-FS}.
 
 \begin{lemma}\label{reverse-maps}
 Suppose $(\A_n, \phi_n^m)$ is a left-invertible sequence of $C^*$-algebras with a compatible left inverse $(\pi_n^m)$ and $\A =\varinjlim (\A_n, \phi_n^m) $. Then for every $n$ there are surjective homomorphisms $\pi_n^\infty: \A \twoheadrightarrow \A_n$ such that $\pi_n^\infty \circ \phi_n^\infty  = \id_{\A_n}$ and $ \pi_n^m \circ \pi_m^\infty  = \pi_n^\infty$ for each $n\leq m$.
\end{lemma}
\begin{proof}  First define $\pi_n^\infty$ on $\bigcup_i \phi_{i}^\infty [\A_i]$, which is dense in $\A$. If $a= \phi_m^\infty(a_m)$ for some $m$ and $a_m \in \A_m$, then let
$$
\pi_n^\infty(a) = \begin{cases}
\pi_n^m (a_m) &\mbox{if } n \leq m \\
\phi_m^n(a_m)  &\mbox{if } n >m
\end{cases}
$$
These maps are well-defined (norm-decreasing) homomorphism, so they extend to $\A$ and satisfy the requirements of the lemma. 
\end{proof}
 
 In particular, each $\A_n$ or any retract of it, is a retract of $\A$. The converse of this is also true.
 
  \begin{lemma}\label{sub-retract}
Suppose $(\A_n, \phi_n^m)$ is a left-invertible sequence of finite-dimensional $C^*$-algebras with   $\A =\varinjlim (\A_n, \phi_n^m) $. 
\begin{enumerate} 
\item If $\D$ is a finite-dimensional subalgebra of $\A$, then $\D$ is contained in an  inner retract of $\A$.
\item  If $\D$ is a finite-dimensional  retract of $\A$, then there is $m\in \mathbb N$ such that $\D$ is a  retract of $\A_{m'}$ for every $m'\geq m$. 
\end{enumerate}
 \end{lemma}
 \begin{proof} 
 Let $(\pi_n^m)$ be a compatible left inverse of $(\A_n, \phi_n^m)$.
 
 (1) If  $\D$ is a finite-dimensional subalgebra of $\A$, then for some $m\in \mathbb N$ and a unitary $u\in \widetilde \A$, it is contained in  $u \phi_m^\infty [\A_m]u^*$ (Lemma~\ref{twisting}). The latter is an  inner retract of $\A$. 
  
  (2) If $\D$ is a retract of $\A$, there is an embedding $\phi:\D \hookr \A$ with a left inverse  $\pi: \A\twoheadrightarrow\D$.  Find $m$ and a unitary $u$ in $\widetilde \A$ such that $u^* \phi[\D] u \subseteq \phi_m^\infty [\A_m]$. 
    This implies that 
    $$ \phi_m^\infty\circ \pi_m^\infty(u^*\phi(d) u) = u^*\phi(d) u$$
    for every $d\in\D$.
  Define $\psi: \D \hookrightarrow \A_m$ by $\psi(d) = \pi_m^\infty (u^*\phi(d) u)$. 
  Then $\psi$ has a left inverse $\theta: \A_m \twoheadrightarrow \D$ defined by $\theta(x) =\pi(u\phi_m^\infty(x) u^*)$, since for every $d\in\D$ we have 
 $$
 \theta(\psi(d)) = \theta(\pi_m^\infty (u^*\phi(d) u)) =\pi(u \phi_m^\infty\circ \pi_m^\infty(u^*\phi(d) u) u^*)
= \pi(\phi(d))=d.
 $$
 
  Because $\A_{m}$ is a retract of $\A_{m'}$, for every $m' \geq m$, we conclude that $\D$ is also a retract of $\A_{m'}$.
\end{proof}


\begin{remark}\label{nonexamples-remark}
It is not surprising that many AF-algebras are not limits of left-invertible sequences of finite-dimensional $C^*$-algebras. This is because, for instance, such an AF-algebra has infinitely many ideals (unless it is finite-dimensional), and admits finite traces, as it maps onto finite-dimensional $C^*$-algebras. Therefore, for example $\mathcal K(\ell_2)$, the $C^*$-algebra of all compact operators on $\ell_2$, and infinite-dimensional UHF-algebras are not limits of left-invertible sequences of finite-dim\-en\-sio\-nal $C^*$-algebras.
Recall that a $C^*$-algebra is \emph{stable} if its tensor product with $\mathcal K(\ell_2)$ is isomorphic to itself. Blackadar's characterization of stable AF-algebras \cite{Blackadar-AF} (see also \cite[Corollary 1.5.8]{Rordam-Stormer}) states that a separable AF-algebra $\A$ is stable if and only if no nonzero ideal of $\A$ admits a nonzero finite (bounded) trace. Therefore no stable AF-algebra is the limit of a left-invertible sequence of finite-dimensional $C^*$-algebras. 
\end{remark}

The following proposition gives another criteria to distinguish these AF-algebras. For example, it can be directly used to show that infinite-dimensional UHF-algebras are not limits of left-invertible sequences of finite-dimensional $C^*$-algebras.

 \begin{proposition}\label{Theorem III.3.5}
Suppose $\A$ is an AF-algebra isomorphic to the limit of a left-invertible sequence of finite-dimensional $C^*$-algebras and $\A =  \overline{\bigcup_n \B_n}$ for an increasing sequence $(\B_n)$ of finite-dimensional subalgebras. Then there is an increasing sequence $(n_i)$ of natural numbers and an increasing sequence $(\C_{i})$ of finite-dimensional subalgebras of $\A$ such that $\A =  \overline{\bigcup_n \C_n}$ and $\B_{n_i}\subseteq \C_{i} \subseteq \B_{n_{i+1}}$ and $\C_{i}$ is an inner retract of $\C_{i+1}$ for every $i\in \mathbb N$.
 \end{proposition}
 \begin{proof}
 Suppose $\A$ is the limit of a left-invertible direct sequence $(\A_n, \phi_n^m)$ of finite-dimensional $C^*$-algebras.
Theorem III.3.5  of \cite{Davidson}, applied to sequences $(\B_n)$ and $(\phi_n^\infty[\A_n])$, shows that there are sequences $(n_i)$, $(m_i)$ of natural numbers and a unitary $u\in \widetilde A$ such that 
$$\B_{n_i}\subseteq u^*\phi_{m_i}^\infty[\A_{m_i}]u \subseteq \B_{n_{i+1}}$$
for every $i\in \mathbb N$.
Let $\C_i =  u^*\phi_{m_i}^\infty[\A_{m_i}]u $.
 \end{proof}

However, the next proposition shows that any AF-algebra is a quotient of an AF-algebra which is the limit of a left-invertible sequence of finite-dimensional $C^*$-algebras.

\begin{proposition}\label{ess-quotient}
 For every  (unital) AF-algebra $\B$ there is a  (unital) AF-algebra $\A\supseteq \B$ which is the limit of a (unital) left-invertible sequence of finite-dimensional $C^*$-algebras and $\A/\J\cong \B$ for an  essential ideal $\J$ of $\A$.
\end{proposition}
\begin{proof}
Suppose $\B$ is the limit of the sequence $(\B_n, \psi_n^m)$ of finite-dimensional $C^*$-algebras and  homomorphisms. Let $\A$ denote the limit of the following diagram:
\begin{equation*}\label{diag2}
\begin{tikzcd}[row sep=small, column sep=large]
\B_1 \arrow[r,  "\psi_1^2"] \arrow[rd, "id"] 
& \B_{2} \arrow[rd, "id"]  \arrow[r,  "\psi_{2}^{3}"]  & \B_{3} \arrow[rd,  "id"] \arrow[r,  "\psi_{3}^{4}"] & \B_{4} \,  {\dots \atop \begin{turn}{-30} \dots \end{turn}}\\
& \B_{1} \arrow[rd, "id"]  & \B_{2} \arrow[rd, "id"] & \B_{3} \,  \begin{turn}{-30} \dots \end{turn}& \\   
& & \B_{1} \arrow[rd, "id"]  &   \B_{2} \,  \begin{turn}{-30} \dots \end{turn}& \\
&&& \B_1 \,  \begin{turn}{-30} \dots \end{turn}&  \\
 & & & & \null
\end{tikzcd}
\end{equation*}

Then $\A$ is an AF-algebra which contains $\B$ and  the connecting maps are  left-invertible embeddings. The ideal $\J$ corresponding to the (directed and hereditary) subdiagram of the above diagram which contains all the nodes except the ones on the top line is essential and clearly $\A/\J \cong \B$.
\end{proof}

\section{AF-algebras with the Cantor property} \label{CP-section}
 
We define the notion of the ``Cantor property" for an AF-algebra. These algebras have properties which are, in a sense, generalizations of the ones satisfied (some trivially) by $C(2^\mathbb N)$.  It is easier to state these properties using the notation for Bratteli diagrams that we fixed in Section \ref{pre-sec}.
For example, every node of the
Bratteli diagram of $C(2^\mathbb N)$ splits in two, which here is generalized to ``each node splits into at least two nodes with the same dimension at some further stage", which of course guarantees that there are no minimal projections in the limit algebra.
\begin{definition}\label{Bratteli}
We say an AF-algebra $\A$ has \emph{the Cantor property} if  there is a sequence $(\A_n, \phi_n^m)$  of finite-dimensional $C^*$-algebras and embeddings such that $\A = \varinjlim (\A_n, \phi_n^m)$ and the Bratteli diagram  $\mathfrak D$  of $(\A_n, \phi_n^m)$ has the following properties:
 \begin{enumerate}
  
\item[(D0)] For every $(n,s)\in \mathfrak D$ there is $(n+1,t)\in \mathfrak D$ such that $\dim(n,s) = \dim(n+1,t)$ and  $(n,s)\to (n+1,t)$.
\item[(D1)] For every $(n,s)\in \mathfrak D$ there are distinct nodes $(m,t),(m,t')\in \mathfrak D$, for some $m>n$, such that $\dim(n,s) = \dim(m,t) = \dim(m,t')$ and  $(n,s)\to (m,t)$ and $(n,s) \to (m,t')$.
\item[(D2)]  For every $(n,s_1), \dots, (n,s_k), (n', s')\in \mathfrak D$ and $\{x_1, \dots , x_k\}\subseteq \mathbb N$ such that $\sum_{i=1}^k x_i\dim(n, s_i) \leq \dim(n', s')$, there is $m\geq n$ such that for some $(m,t)\in \mathfrak D$  we have $\dim(m,t) = \dim(n',s')$ and there are exactly  $x_i$ distinct paths from $(n, s_i)$ to $(m,t)$ in $\mathfrak D$.
 \end{enumerate}
\end{definition}
The Bratteli diagram of $C(2^\mathbb N)$ trivially satisfies these conditions and therefore $C(2^\mathbb N)$ has the Cantor property.
 \begin{remark}{\label{remark-CP-def}} Condition (D0) states that
  $(\A_n, \phi_n^m)$ is a left-invertible sequence.
Dropping  (D0) from Definition \ref{Bratteli} does not change the definition (i.e., $\A$ has the Cantor property if and only if it has a representing sequence satisfying (D1) and (D2)). This is because (D1) alone implies the existence of a left-invertible sequence with limit $\A$ that still satisfies (D1) and (D2).   In fact, $\A$ has the Cantor property if and only if any representing sequence satisfies (D1) and (D2). However, we add (D0) for simplicity to make sure that $(\A_n, \phi_n^m)$ is already a left-invertible sequence, since, as we shall see later, being the limit of a left-invertible direct sequence of finite-dimensional $C^*$-algebras is a crucial and helpful property of AF-algebras with the Cantor property. 
 Condition (D2) can be rewritten as 
\begin{enumerate}
\item[(D$2'$)] For every ideal $\D$ of some $\A_n$,  if $M_\ell$ is a retract of $\A$ and  $\gamma : \D \hookr M_\ell$ is an embedding, then  there is $m\geq n$ and  $\A_{m,t}\subseteq \A_m$ such that $\A_{m,t} \cong M_\ell$  and  $\Mult_{\phi_n^m}(\D , \A_{m,t} )  =  \Mult_{\gamma}(\D, M_\ell)$.
 \end{enumerate}
\end{remark}

Definition \ref{Bratteli} may be adjusted for unital AF-algebras where all the maps are considered to be unital.
 \begin{definition}\label{CS-unital-def}
A unital AF-algebra $\A$ has the Cantor property if and only if  it satisfies the conditions of  Definition \ref{Bratteli}, where $\phi_n^m$ are unital and in condition (D2) the inequality $\sum_{i=1}^k x_i \dim(n, s_i) \leq \dim(n', s')$ is replaced with equality.
 \end{definition}

 \begin{proposition}\label{sum-retract}
 Suppose $\A$  is an AF-algebra with  Cantor property. If $\D, \E$ are finite-dimensional retracts of $\A$, then so is $\D\oplus \E$.
 \end{proposition}
\begin{proof}
Suppose $\D= \D_1 \oplus \D_2 \oplus \dots \oplus \D_l$ and $\E= \E_1 \oplus \E_2 \oplus \dots \oplus \E_k$, where $\D_i,\E_i$ are isomorphic to  matrix algebras. By Lemma \ref{sub-retract} both $\D$ and $\E$ are retracts of  some $\A_m$, which means that all $\D_i$ and $\E_i$ appear in $\A_m$ as  retracts (ideals). By (D1) and enlarging $m$, if necessary, we can make sure these retracts in $\A_m$ are orthogonal, meaning that $\A_m \cong \D \oplus \E \oplus \F$, for some finite-dimensional $C^*$-algebra $\F$. Therefore $\D \oplus \E$ is a retract of $\A_m$ and as a result, it is a retract of $\A$.
\end{proof}

 \begin{lemma}\label{absorbing}
 Suppose $\A$ is an AF-algebra with  the Cantor property, witnessed by $(\A_n, \phi_n^m)$ satisfying Definition \ref{Bratteli} and $\E$ is a finite-dimensional  retract of $\A$. If $\gamma: \A_n \hookr \E$ is a left-invertible embedding then there are $m\geq n$ and a left-invertible embedding $\delta: \E \hookr \A_m$ such that $\delta\circ  \gamma = \phi_n^m$.
 \end{lemma}
 \begin{proof} 
Suppose    $\A_n  = \A_{n,1} \oplus \dots \oplus \A_{n,l}$ and  $\E= \E_1 \oplus \E_2 \oplus \dots \oplus \E_k$ where $\E_i$ and $\A_{n,j}$ are all matrix algebras. Let $\pi_i$ denote the canonical projection from $\E$ onto $\E_i$. For every $i\leq k $ put
$$Y_i = \{j \leq l :\gamma[\A_{n,j}]\cap \E_i \neq 0\},$$
and let $\A_{n,Y_i} = \bigoplus_{j\in Y_i} \A_{n,j}$.
Then $\A_{n,Y_i}$ is an ideal (a retract) of $\A_n$ and the map $\gamma_i : \A_{n,Y_i} \hookr \E_i$, the restriction of $\gamma$ to $\A_{n,Y_i}$ composed with $\pi_i$,  is an embedding. Since $\E$ is a finite-dimensional retract of $\A$, it is a retract of some $\A_{n^\prime}$ (Lemma \ref{sub-retract}). So each $\E_i$ is a retract of $\A_{n^\prime}$.  By applying (D2) for each $i\leq k$ there are $m_i\geq n $ and $(m_i,t_{i})  \in \mathfrak D$ such that $\dim(m_i, t_{i}) = \dim(\E_i)$ and 
$\Mult_{\phi_n^{m_i}}(\A_{n,Y_i} , \A_{m_i, t_i}) =  \Mult_{\gamma_i}(\A_{n,Y_i}, \E_i)$.  Let $m =  \max\{m_i : i\leq k\}$ and by (D0)  find $(m,s_i)$  such that $\dim(m_i, t_i) = \dim(m,s_i)$ and $(m_i, t_i) \to (m,s_i)$. Applying (D1)  and possibly increasing $m$ allows us to make sure that $(m,s_i) \neq (m,s_j)$ for distinct $i,j$ and therefore   $\A_{m, s_{i}} $ are pairwise orthogonal.  Then $\{\A_{m,s_i} : i\leq k\}$ is a sequence of pairwise orthogonal subalgebras (retracts) of $\A_m$ such that  $\A_{m,s_i} \cong \E_i$ and  $$\Mult_{\phi_n^{m}}(\A_{n,Y_i} , \A_{m,s_i}) = \Mult_{\gamma_i}(\A_{n,Y_i}, \E_i).$$
 By Lemma \ref{matrix-absorbing} 
  there are isomorphisms $\delta_i : \E_i \hookr \A_{m,s_i}$ such that $\gamma_i \circ \delta_i $ is equal to the restriction of $\phi_n^m$ to $\A_{n,Y_i}$ projected onto $\A_{m,s_i}$.     
  
  Suppose $1_m$ is the unit of $\A_m$ and $q_i$ is the unit of $\A_{m,s_i}$.   Each $q_i$ is a central projection of $\A_m$, because  $\A_{m,s_i}$ are ideals of $\A_m$.
 Since $\gamma$ is left-invertible,  for each $j\leq l$ there is $k(j) \leq k$ such that $\A_{n,j}\cong  \E_{k(j)}$ and  $\hat\gamma_{j} = \pi_{k(j)}\circ \gamma|_{\A_{n,j}}$ is an isomorphism. Also for $j\leq l$ let 
   $$X_j = \{i\leq k : \gamma[\A_{n,j}] \cap \E_i \neq 0\}.$$
   Note that 
   \begin{enumerate}
\item   $k(j) \in X_j$,
  \item $k(j') \notin X_j$ if $j\neq j'$,
  \item $i\in X_j \Leftrightarrow j\in Y_i$.
  \end{enumerate}
  Let $\hat\delta_j:  \E_{k(j)} \to (1_m -  \sum_{i\in X_j} q_i) \A_m (1_m - \sum_{i\in X_j} q_i)$ be the  homomorphism defined by 
   $$\hat\delta_j (e)  =  (1_m - \sum_{i\in X_j}q_i)\phi_n^m(\hat\gamma_j^{-1}(e)) (1_m - \sum_{i\in X_j} q_i).$$
 Define  $\delta: \E \hookr \A_m $ by $$\delta(e_1, \dots ,e_k) = \hat\delta_1(e_{k(1)}) + \dots + \hat\delta_l(e_{k(l)}) + \delta_1 (e_1) + \dots +\delta_k (e_k).$$
 Since each $\delta_i$ is an isomorphism, it is clear that $\delta$ is left-invertible.
  To check that $\delta \circ \gamma = \phi_n^m$, by linearity of the maps it is enough to check it only for $\bar a = (0, \dots, 0,a_j, 0, \dots,0)\in \A_n$. If $\gamma(\bar a) = (e_1, \dots , e_k)$ then 
 $$
  e_i = \begin{cases} 0 & i\notin X_j \\
  \gamma_i(\bar a)      & i \in X_j
   \end{cases}
  $$
for $i\leq k$. Also note that $e_{k(j)}=\hat \gamma_j(a_j)$. Assume $X_j = \{r_1, \dots, r_\ell\}$. Then by (1)-(3) we have 
\begin{align*}
\delta \circ \gamma (\bar a) & = \hat \delta_j (\hat \gamma_j (a_j))+ \delta_{r_1}(\gamma_{r_1}(\bar a))+ \dots+ \delta_{r_\ell}(\gamma_{r_\ell}(\bar a)) \\
&=   (1_m - \sum_{i\in X_j}q_i)\phi_n^m(\bar a) (1_m - \sum_{i\in X_j} q_i) + q_{r_1} \phi_n^m(\bar a) q_{r_1} +\dots +  q_{r_\ell} \phi_n^m(\bar a) q_{r_\ell} \\
&= \phi_n^m(\bar a).
\end{align*}
This completes the proof.
\end{proof}

\subsection{AF-algebras with the Cantor property are $C(2^\mathbb N)$-absorbing}
 Suppose $\A$ is an AF-algebra with  Cantor property.
Define $\A^\mathcal C$ to be the limit of the sequence $(\B_n, \psi_n^m)$ such that $\B_n = \bigoplus_{i \leq 2^{n-1}} \A_{n} \cong \mathbb C^{2^{n-1}} \otimes \A_n$ and $\psi_n^{n+1} = \bigoplus_{i \leq 2^n} \phi_n^{n+1}$, as shown in the following diagram
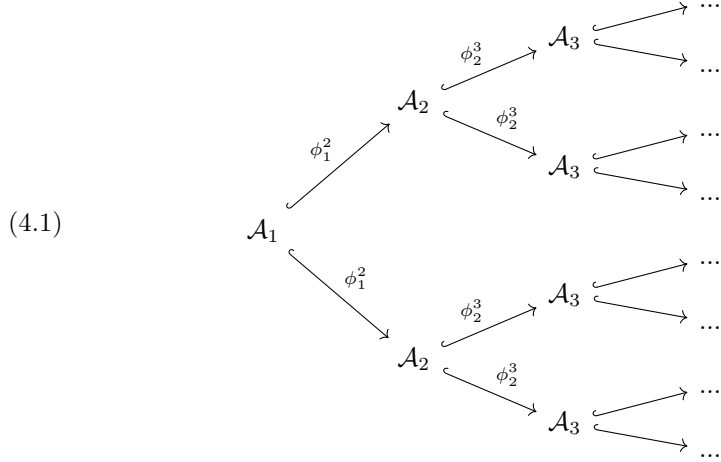
\begin{equation}\label{A^N-diag}
\begin{tikzcd} [row sep=0.05mm , column sep=large] 
& & & ... \\
& & \A_{3}  \arrow[ru, hook] \arrow[rd, hook] \\
&&& ...\\
& \A_{2} \arrow[ruu, hook, "\phi_2^3"]   \arrow[rdd, hook, "\phi_2^3"]  & & \\
&&& ...\\
& & \A_{3}  \arrow[ru, hook]  \arrow[rd, hook] \\
&&& ...\\
\A_1 \arrow[ruuuu, hook, "\phi_1^2"]  \arrow[rdddd, hook, "\phi_1^2"] \\
&&& ...\\
& & \A_{3}  \arrow[ru, hook]  \arrow[rd, hook] \\
&&& ...\\
& \A_{2} \arrow[ruu, hook, "\phi_2^3"]   \arrow[rdd, hook, "\phi_2^3"]  & & \\
&&& ...\\
& & \A_{3}  \arrow[ru, hook] \arrow[rd, hook] \\
&&& ...\\
\end{tikzcd}
\end{equation}
It is straightforward to check that $\A^\mathcal C \cong \A \otimes C(2^\mathbb N)\cong C(2^\mathbb N, \A)$.
\begin{lemma}\label{A^N}
$\A^\mathcal C$ has the Cantor property.
\end{lemma}
\begin{proof}
We check that $(\B_n, \psi_n^m)$ satisfies (D0)--(D2).
Each $\psi_n^{n+1}$ is left-invertible, by Proposition \ref{fact1} and since $\phi_n^{n+1}$ is  left-invertible, therefore (D0) holds. Conditions (D1) and (D2) are trivially satisfied by analyzing the Bratteli diagram (\ref{A^N-diag}), since $\A$ satisfies them. 
\end{proof}

\begin{lemma}\label{A^N=A}
Suppose $\A$ is an AF-algebra with  Cantor property. Then $\A \otimes C(2^\mathbb N)$ is isomorphic to $\A$.
\end{lemma}
\begin{proof} Identify $\A \otimes C(2^\mathbb N)$ with $\A^\mathcal C$.
Find sequences $(m_i)$ and $(n_i)$ of natural numbers and  left-invertible embeddings $\gamma_i : \A_{n_i} \hookr \B_{m_{i+1}}$ and 
$\delta_i : \B_{n_i} \hookr \A_{m_{i}}$ such that  $n_1 = m_1 = 1$, $m_2 = 2$ and $\gamma_1 = \psi_1^2$ and the diagram below is commutative.
\begin{equation}\label{intertwining}
\begin{tikzcd}
\B_1 \arrow[r, hook, "\psi_1^2"] \ar[d,equal]
& \B_{m_2} \arrow[rd, hook, "\delta_2"]  \arrow[rr, hook,  "\psi_{m_2}^{m_3}"] & & \B_{m_3} \arrow[rd, hook, "\delta_2"] \arrow[rr, hook,  "\psi_{m_3}^{m_4}"] &  & \dots & \A^\mathcal C \arrow[d, "\phi"] \\
\A_{1} \arrow[rr, hook,  "\phi_1^{n_2}"] \arrow[ru, hook,  "\gamma_1"] &
& \A_{n_2} \arrow[rr, hook,  "\phi_{n_2}^{n_3}"] \arrow[ru, hook,  "\gamma_2"] & & \A_{n_3} \arrow[r,  hook,  "\phi_{n_3}^{n_4}"]\arrow[ru, hook,  "\gamma_3"] & \dots & \A
\end{tikzcd}
\end{equation}
 The existence of such $\gamma_i$ and $\delta_i$ is guaranteed by Lemma \ref{absorbing}, since each $\B_i$ is a  retract of $\A$, by Lemma \ref{sub-retract} and Proposition \ref{sum-retract}, and of course each $\A_i$ is a retract of $\B_i$.  The universal property of inductive limits implies the existence of an isomorphism between $\A$ and $\A^\mathcal C$.
\end{proof}

\begin{remark}
As we will see in section \ref{tensor-cp}  the tensor products of two AF-algebras with Cantor property do not necessarily have the Cantor property. 
\end{remark}

\subsection{ Ideals}
 Let $\A = \varinjlim_n (\A_n, \phi_n^m)$ be an AF-algebra with  Cantor property, such that the Bratteli diagram $\mathfrak D$ of $(\A_n,\phi_n^m)$ 
satisfies (D0)--(D2) of Definition \ref{Bratteli}.  Let $\mathfrak J\subseteq \mathfrak D$ denote the Bratteli diagram of an ideal $\J\subseteq \A$. Put $\J_n = \bigoplus_{(n,s)\in \mathfrak J} \A_{n,s}$, which is an  ideal (a retract) of $\A_n$. Then $\J = \varinjlim_n (\J_n, \phi_n^m|_{\J_n})$. It is automatic from the fact that $\mathfrak J$ is a directed subdiagram of $\mathfrak D$ that each $\phi_n^m|_{\J_n}: \J_n \hookr \J_m$ is left-invertible and that $(\J_n, \phi_n^m|_{\J_n})$ satisfies (D0)--(D2). In particular:

\begin{proposition}
 Any ideal of an AF-algebra with  Cantor property also  has the Cantor property.    
\end{proposition}

Here is another elementary fact about $C(2^\mathbb N)$ that is (essentially by Lemma \ref{A^N=A}) passed on to AF-algebras with  Cantor property.

\begin{proposition}\label{prop} Suppose $\A$ is an AF-algebra with  Cantor property and $\Q$ is  a quotient of $\A$. Then there is a surjection $\eta: \A \surj \Q$  such that $ker(\eta)$ is an essential ideal of $\A$.
\end{proposition}
\begin{proof}
It is enough to show that there is an essential ideal $\J$ of $\A$ such that $\A/\J$ is isomorphic to $\A$. 
In fact, we will show that there is an essential ideal $\J$ of $\A^\mathcal C$ such that $\A^\mathcal C/ \J$ is isomorphic to $\A$. This is enough since $\A^\mathcal C$ is isomorphic to $\A$ (Lemma \ref{A^N=A}). Let $\mathfrak D$ be the Bratteli diagram of $\A^\mathcal C$ as in Diagram (\ref{A^N-diag}). Let $\mathfrak J$ be the directed and hereditary subdiagram of $\mathfrak D$ containing all the nodes in Diagram (\ref{A^N-diag}) except the lowest line.  Being directed and hereditary, $\mathfrak J$ corresponds to an ideal $\J$, which intersects any other directed and hereditary subdiagram of $\mathfrak D$. Therefore $\J$ is an essential ideal of $\A^\mathcal C$ and $\A^\mathcal C/\J$ is isomorphic to the limit of the sequence $\A_1 \xrightarrow{\phi_1^2} \A_{2} \xrightarrow{\phi_2^3} \A_{3}\xrightarrow{\phi_3^4} \dots $ in the lowest line of Diagram (\ref{A^N-diag}), which is $\A$. 
\end{proof}

\section{Fra\"\i ss\'e categories}\label{Flim-intro-section}

Suppose $\KK$ is a category of metric structures with non-expansive (1-Lipschitz) morphisms. We refer to objects and morphisms (arrows) of $\KK$ by $\KK$-objects and $\KK$-arrows, respectively. We write $A\in \KK$ if $A$ is a $\KK$-object and $\KK(A, B)$ to denote the set of all $\KK$-arrows from $A$ to $B\in \KK$.  The category $\KK$ is \emph{metric-enriched} or \emph{enriched over metric spaces} if for every $\KK$-objects $A$ and $B$ there is a metric $d$ on $\KK(A,B)$ satisfying
$$d(\psi_0 \circ \phi,\psi_1 \circ \phi) \leq d(\psi_0 ,\psi_1) \qquad \text{and} \qquad d(\psi \circ \phi_0,\psi \circ \phi_0) \leq d(\phi_0 ,\phi_1)$$
whenever the compositions make sense.
We say $\KK$ is \emph{enriched over complete metric spaces} if $\KK(A,B)$ is a complete metric space for every $\KK$-objects $A$, $B$.

A \emph{$\KK$-sequence} is a direct sequence in $\KK$, that is, a covariant functor from the category of all positive integers (treated as a poset) into $\KK$.

In our cases, $\KK$ will always be a category of finite-dimensional $C^*$-algebras with left-invertible embeddings. However, we would like to invoke the general theory of Fra\"\i ss\'e categories, which is possibly applicable to other similar contexts.

\begin{definition} \label{Fraisse-def}
We say $\KK$ is a \emph{Fra\"\i ss\'e category} if
\begin{enumerate}
\item[(JEP)] $\KK$ has \emph{the joint embedding property}:  for $A, B \in \KK$ there is $C\in \KK$ such that $\KK(A, C)$ and $\KK(B,C)$  are nonempty.
\item[(NAP)] $\KK$ has \emph{the near amalgamation property}:   for every $\epsilon>0$, objects $A,B,C\in \KK$, arrows $\phi\in \KK(A, B)$ and $\psi\in \KK(A, C)$, there are $D\in \KK$ and $\phi' \in \KK(B,D)$ and $\psi'\in \KK(C,D)$ such that $d(\phi'\circ \phi , \psi' \circ \psi)<\epsilon$.
 \item[(SEP)] $\KK$ is \emph{separable}:  there is a countable \emph{dominating} subcategory $\mathfrak C$, that is, 
 \begin{itemize}
     \item for every $A\in \KK$ there is $C\in \mathfrak C$ and a $\KK$-arrow $\phi: A \to C$,
     \item for every $\epsilon >0$ and a $\KK$-arrow $\phi: A \to B$ with $A\in \mathfrak C$, there exist a $\KK$-arrow $\psi: B \to C$ with $\C \in \mathfrak C$ and a $\mathfrak C$-arrow $\alpha: A \to C$ such that $d(\alpha, \psi \circ \phi)< \epsilon$.
 \end{itemize}
\end{enumerate}
\end{definition}

Now suppose that $\KK$ is contained in a bigger metric-enriched category $\LL$ so that every sequence in $\KK$ has a limit in $\LL$.
We say that $\KK \subseteq \LL$ has the \emph{almost factorization property} if given any sequence $(X_n,f_n^m)$ in $\KK$ with limit $X_\infty$ in $\LL$, for every $\epsilon>0$, for every $\LL$-arrow $g \colon A \to X_\infty$ with $A \in \KK$ there is a $\KK$-arrow $g' \colon A \to X_n$ for some positive integer $n$, such that $d(f_n^\infty \circ g', g) \leq \epsilon$, where $f_n^\infty \colon X_n \to X_\infty$ comes from the limiting cocone\footnote{Formally, the limit, or rather \emph{colimit} of $(X_n, f_n^m)$ is a pair consisting of an $\LL$-object $X_\infty$ and a sequence of $\LL$-arrows $f_n^\infty \colon X_n \to X_\infty$ satisfying suitable conditions. This sequence is called the (co-)limiting cocone. We use the word ``limit" instead of ``colimit" as we consider only covariant functors from the positive integers, called \emph{sequences}.}. 
\begin{theorem}{\cite[Theorem 3.3]{Kubis-ME}}\label{generic}
Suppose  $\KK$ is a   Fra\"\i ss\'e category.
Then there exists a sequence $(U_n, \phi_n^m)$ in $\KK$ satisfying 
\begin{itemize}
    \item[{\rm(F)}] for every $n\in \mathbb N$, for every $\epsilon>0$ and for every $\KK$-arrow $\gamma: U_n \to D$, there are $m\geq n$ and a $\KK$-arrow $\delta: D \to U_m$ such that $d(\phi_n^m , \delta \circ \gamma)<\epsilon$.
\end{itemize}
\end{theorem}

If $\KK$ is a Fra\"\i ss\'e category, the $\KK$-sequence $(U_n, \phi_n^m)$ from Theorem \ref{generic} is uniquely determined by the ``Fra\"\i ss\'e condition" (F). That is, any two $\KK$-sequences satisfying (F) can be approximately intertwined (there is an approximate back-and-forth between them), and hence the limits of the sequences (typically in a bigger category containing $\KK$) must be isomorphic (see \cite[Theorem 3.5]{Kubis-ME}). Therefore the $\KK$-sequence satisfying (F) is usually referred to as ``the" \emph{Fra\"\i ss\'e sequence}. The limit of the Fra\"\i ss\'e sequence is called  the \emph{Fra\"\i ss\'e limit} of the category $\KK$.
In our case, $\KK$ will be a category of finite-dimensional $C^*$-algebras and the limit is just the inductive limit (also called colimit) in the category of all (or just separable) $C^*$-algebras.

\begin{theorem}[{cf. \cite{Kubis-ME}}]\label{uni-hom}
Assume $\KK$ is a Fra\"\i ss\'e category contained in a category $\LL$ such that every sequence in $\KK$ has a limit in $\LL$ and every $\LL$-object is the limit of some sequence in $\KK$.
Let $U\in \LL$ be the Fra\"\i ss\'e limit of $\KK$. Then 
\begin{itemize}[leftmargin=*]
\item (uniqueness) $U$ is unique, up to isomorphisms.
\item(universality) For every $\LL$-object $B$ there is an $\LL$-arrow $\phi: B \to U$.
\end{itemize}
Furthermore, if $\KK \subseteq \LL$ has the almost factorization property then
\begin{itemize}[leftmargin=*]
	\item(almost $\KK$-homogeneity) For every $\epsilon>0$, $\KK$-object $A$ and $\LL$-arrows $\phi_i : A \to U$ ($i=0,1$), there is an automorphism $\eta: U \to U$ such that $d(\eta \circ \phi_0  , \phi_1) <\epsilon$.
\end{itemize}
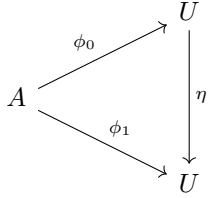
$$\begin{tikzcd}
& & U \ar[dd, "\eta"] \\
A \ar[rru, "\phi_0"] \ar[rrd, "\phi_1"] & & \\
& & U
\end{tikzcd}$$
\end{theorem}

\begin{definition}
Let $\ddagger\KK$ denote the category with the same objects as $\KK$, but a $\ddagger \KK$-arrow from $A$ to $B$ is a pair $(\phi, \pi)$ where $\phi, \pi$ are $\KK$-arrows, $\phi:A \to B$ is left-invertible and  $\pi: B \to A$ is a left inverse of $\phi$. We will denote such $\ddagger \KK$-arrow by $(\phi, \pi): A \to B$. The composition is $(\phi, \pi) \circ (\phi', \pi') = (\phi \circ \phi', \pi' \circ \pi)$. The category $\ddagger \KK$ is usually called the category of \emph{embedding-projection pairs} or briefly EP-pairs over $\KK$ (see \cite{Kubis-FS}).
\end{definition}

\begin{definition}
We say $\ddagger \KK$ has the \emph{near proper amalgamation property} if 
for every $\epsilon>0$, objects $A, B, C\in \KK$, arrows $(\phi, \pi)\in \ddagger\KK(A, B)$ and $(\psi, \theta)\in \ddagger\KK(A, C)$, there are $D\in \ddagger\KK$ and $(\phi', \pi') \in \ddagger\KK(B, D)$ and $(\psi', \theta')\in \ddagger\KK(C, D)$ such that the diagram 
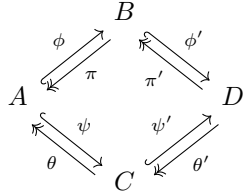
\begin{equation*}\label{diag-PAP}
\begin{tikzcd}
 &  B \arrow[dr, hook',   "\phi'"] \arrow[dl, shift left=1ex, two heads, "\pi"]  \\ 
A \arrow[ur, hook', "\phi"] \arrow[dr,  hook,  "\psi"] 
& & D \arrow[ul, shift left=1ex, two heads, "\pi'"] \arrow[dl,shift left=1ex, two heads, "\theta'"]\\
& C \arrow[ur, hook, "\psi^\prime "]   \arrow[ul, shift left=1ex, two heads, "\theta"]
\end{tikzcd}
\end{equation*}
``fully commutes" up to $\epsilon$, meaning that  
$d( \phi' \circ \phi , \psi' \circ \psi )$, $d( \pi' \circ \pi , \theta' \circ \theta )$,  $d( \phi \circ \theta , \pi' \circ \psi' )$ and 
$d( \psi \circ \pi , \theta' \circ \phi' )$ are all less than or equal to $\epsilon$.
We say $\ddagger \KK$ has the ``proper amalgamation property" if $\epsilon$ could be $0$.
\end{definition}

Let us denote by $\dagger \KK$ the category of left-invertible $\KK$-arrows. In other words, $\dagger \KK$ is the image of $\ddagger \KK$ under the functor that forgets the left inverse, namely mapping $(\phi, \pi)$ to $\phi$. Note that all three categories $\KK$, $\dagger \KK$, and $\ddagger \KK$ have the same objects.

\begin{lemma}\label{EPlim-retract}
	Suppose $\LL$ is enriched over complete metric spaces, $\dagger \KK$ is a Fra\"\i ss\'e category with Fra\"\i ss\'e limit $U$, and $\ddagger \KK$ has the proper amalgamation property. Then for every $\LL$-object $B$ isomorphic to the limit of a sequence in $\dagger \KK$ there is a pair of $\LL$-arrows $\alpha \colon B \to U$, $\beta \colon U \to B$ such that
	$$\beta \circ \alpha = \id_{B}.$$
\end{lemma}

\begin{proof}
	Suppose $(U_n, \phi_n^m)$ is a Fra\"\i ss\'e sequence in $\dagger \KK$. Suppose first that the sequence satisfies (F) with $\epsilon=0$ and that $\ddagger \KK$ has the proper amalgamation property, namely with $\epsilon=0$ (this will be the case in the next section). In this case we do not use the fact that $\LL$ is enriched over complete metric spaces.
	
	Fix a $\dagger \KK$-sequence $(B_n, \psi_n^m)$ whose direct limit is $B$. For each $n$ we may choose a left inverse $\theta_n^{n+1}$ to $\psi_n^{n+1}$ and next, setting $\theta_n^m = \theta_{m-1}^m \circ \dots \circ \theta_n^{n+1}$ for every $n<m$, we obtain a $\ddagger \KK$-sequence $(B_n, (\psi_n^m, \theta_n^m))$ whose direct limit is $B$.
	Using (JEP) of $\dagger \KK$ and fixing arbitrary left inverses, find $F_1\in \KK$ and $\ddagger \KK$-arrows $(\gamma_1, \eta_1): U_1 \to F_1$ and $(\mu_1, \nu_1): B_1 \to F_1$. By (F) and again fixing arbitrary left inverses, there are $n_1 \geq 1$ and a $\ddagger \KK$-arrow $(\delta_1, \lambda_1): F_1 \to U_{n_1}$ such that 
	$\phi_1^{n_1} = \delta_1 \circ \gamma_1$ (see Diagram (\ref{diag-retract}) below).   
	
	Consider the composition arrow $(\delta_1 \circ \mu_1, \nu_1 \circ \lambda_1): B_1 \to U_{n_1}$ and $(\psi_1^2, \theta_1^2): B_1 \to B_2$ and use the proper amalgamation property to find $\F_2\in \KK$ and $\ddagger \KK$-arrows $(\mu_2, \nu_2): B_2 \to F_2$ and $(\gamma_2, \eta_2): U_{n_1} \to F_2$ such that 
	\begin{equation}\label{eq2}
		\gamma_2 \circ \delta_1 \circ \mu_1 = \mu_2 \circ \psi_1^2  \qquad \text{and} \qquad  \nu_2 \circ \gamma_2 = \psi_1^2 \circ \nu_1 \circ \lambda_1.
	\end{equation}
	Again using (F) we can find $n_2 \geq n_1$ and $(\delta_2, \lambda_2): F_2 \to U_{n_2}$ such that 
	\begin{equation}\label{eq3}
		\phi_{n_1}^{n_2} = \delta_2 \circ \gamma_2.
	\end{equation}
	Combining equations in (\ref{eq2}) and (\ref{eq3}) we have (also can be easily checked in Diagram (\ref{diag-retract})):
	\begin{equation*}\label{eq4}
		\phi_{n_1}^{n_2} \circ \delta_1 \circ \mu_1 = \delta_2 \circ \mu_2 \circ \psi_1^2 \qquad \text{and} \qquad  
		\psi_{1}^{2} \circ \nu_1 \circ \lambda_1  = \nu_2 \circ \lambda_2 \circ \phi_{n_1}^{n_2}.
	\end{equation*}
	Again use the proper amalgamation property to find $F_3\in \KK$ and $(\mu_3, \nu_3): B_3 \to F_3$ and $(\gamma_3, \eta_3): U_{n_2} \to F_3$. Follow the procedure, by finding $\ddagger \KK$-arrow $(\delta_3, \lambda_3): F_3 \to U_{n_3}$, for some $n_3\geq n_2$  such that 
	\begin{equation*}
		\phi_{n_2}^{n_3} \circ \delta_2 \circ \mu_2 = \delta_3 \circ \mu_3 \circ \psi_2^3 \qquad \text{and} \qquad  \psi_{2}^{3} \circ \nu_2 \circ \lambda_2 =\nu_3 \circ \lambda_3 \circ \phi_{n_2}^{n_3} 
	\end{equation*}
	\begin{equation}\label{diag-retract}
		\begin{tikzcd}[column sep=normal]
			U_1 \arrow[rr, hook,  "\phi_1^{n_1}"] \arrow[dr, hook, shift left=1ex, "\gamma_1"] &  & U_{n_1} \arrow[rr, hook, "\phi_{n_1}^{n_2}"] \arrow[dr, hook, shift left=1ex, "\gamma_2"]  \arrow[dl, two heads, "\lambda_1"] &  & U_{n_2} \arrow[rr, hook,  "\phi_{n_2}^{n_3}"] \arrow[dr, hook, shift left=1ex, "\gamma_3"]   \arrow[dl, two heads, "\lambda_2"] & & U_{n_3} \arrow[r, hook,  "\phi_{n_3}^{n_4}"]   \arrow[dl, two heads, "\lambda_3"]  & \dots   \ \  U  \arrow[dd, two heads, shift left=3ex, "\beta"]
			\\
			& F_1 \arrow[dl, two heads, "\nu_1"] \arrow[ul, two heads, "\eta_1"] \arrow[ur, hook, shift left=1ex, "\delta_1"] &  &  F_2 \arrow[dl, two heads, "\nu_2"]  \arrow[ul, two heads, "\eta_2"] \arrow[ur, hook, shift left=1ex, "\delta_2"] & & 
			F_3 \arrow[dl, two heads, "\nu_3"] \arrow[ul, two heads, "\eta_3"] \arrow[ur, hook, shift left=1ex, "\delta_3"]  & &  \\
			B_1 \arrow[rr, hook, shift left=1ex, "\psi_1^2"] \arrow[ur, hook, shift left=1ex, "\mu_1"] &  & B_2 \arrow[rr, hook,shift left=1ex, "\psi_2^3"] \arrow[ur, hook, shift left=1ex, "\mu_2"]  \arrow[ll, two heads, "\theta_1^2"]   &  & B_3 \arrow[rr, hook, shift left=1ex, "\psi_3^4"] \arrow[ur, hook, shift left=1ex, "\mu_3"] \arrow[ll, two heads, "\theta_2^3"]  &  &  \dots \arrow[ll, two heads, "\theta_3^4"]  & \ \ \ \ \ B \arrow[uu, hook, shift right=2ex, "\alpha"]
		\end{tikzcd}
	\end{equation}
	Let $\alpha_i = \delta_i \circ \mu_i$ and $\beta_i: \nu_i \circ \lambda_i$.
	By the construction, for every $i\in \mathbb N$ we have 
	$$\phi_{n_i}^{n_{i+1}} \circ \alpha_i = \alpha_{i+1} \circ \psi_{i}^{i+1} \qquad \text{and} \qquad  \psi_1^2 \circ \beta_i = \beta_{i+1} \circ \phi_{n_i}^{n_{i+1}}$$
	and $\beta_i$ is a left inverse of $\alpha_i$.
	Then $\alpha = \lim_i \alpha_i$ is a well-defined arrow from $B$ to $U$ and $\beta = \lim_i \beta_i$ is a well-defined arrow from $U$ onto $B$ such that $\beta \circ \alpha = \id_{B}$.
	
	Finally, if $\ddagger \KK$ has the near proper amalgamation property and the sequence $(U_n, \phi_n^m)$ satisfies (F) with arbitrary $\epsilon>0$, we repeat the arguments above, except that Diagram~(\ref{diag-retract}) is no longer commutative. On the other hand, at step $n$ we may choose $\epsilon = 2^{-n}$ and then the arrows $\alpha$ and $\beta$ are obtained as limits of suitable Cauchy sequences in $\LL(B, U)$. This is the only place where we need to know that $\LL$ is enriched over complete metric spaces.
\end{proof}

Let us mention that the concept of EP-pairs has been already used by Garbuli\'nska-W\c egrzyn \cite{JGW} in the category of finite dimensional normed spaces, obtaining isometric uniqueness of a complementably universal Banach space.

\section{Categories of finite-dimensional $C^*$-algebras and left-invertible mappings} 
In this section $\KK$ always denotes a (naturally metric-enriched) category whose objects are (not necessarily all) finite-dimensional $C^*$-algebras, closed under isomorphisms, and $\KK$-arrows are left-invertible embeddings. For such $\KK$, let  $\LL\KK$ denote the ``category of limits" of $\KK$; a category whose objects are limits of  $\KK$-sequences and 
 if $\B$ and $\mathcal C$ are $\LL\KK$-objects, then an  $\LL\KK$-arrow from $\B$ into $\mathcal C$ is a left-invertible embedding  $\phi:\B \hookr \mathcal C$.  Clearly $\LL\KK$ contains $\KK$ as a full subcategory. The metric defined between $\LL\KK$-arrows $\phi$ and $\psi$ with the same domain and codomain is $\|\phi - \psi\|$. 
 
 For every such category $\KK$, let $\widehat\KK$ denote the category whose objects are exactly the objects of $\KK$, but the $\widehat\KK$-morphisms are all homomorphisms between the objects. Then we can define the corresponding category of EP-pairs $\ddagger\widehat\KK$ as in the previous section. In what follows, let us agree to write $\ddagger\KK$ instead of $\ddagger\widehat\KK$. Hence, the $\ddagger\KK$-morphisms are of the form $(\phi,\pi)$, where $\phi$ is a $\KK$-morphism and $\pi$ is a homomorphism which is a left inverse of $\phi$.  
 
 \begin{remark}\label{e=0} 
If $\KK$ is a category of finite-dimensional $C^*$-algebras and embeddings, then it has the near amalgamation property (NAP) if and only if it has the amalgamation property (\cite[Lemma 3.2]{Eagle}), namely, with $\epsilon = 0$. Similarly, the near proper amalgamation property of $\ddagger \KK$ is equivalent to the proper amalgamation property of $\ddagger \KK$. Also in this case, the  Fra\"\i ss\'e sequence $(\U_n,\phi_n^m)$, whenever it exists for $\KK$, satisfies the Fra\"\i ss\'e condition (F) of Theorem \ref{generic} with $\epsilon = 0$. Therefore in this section (F) refers to the following condition.
\begin{itemize}
    \item[(F)]  for every $n\in \mathbb N$ and for every $\KK$-arrow $\gamma: \mathcal U_n \to \D$, there are $m\geq n$ and $\KK$-arrow $\delta: \D \to \mathcal U_m$ such that $\phi_n^m = \delta \circ \gamma$.
\end{itemize}
 \end{remark}

\begin{lemma} \label{L1-L3} $\KK \subseteq \LL\KK$ has the almost factorization property.
\end{lemma}
\begin{proof}
Suppose $\B \in \LL\KK$ is the limit of the $\KK$-sequence $(\B_n, \psi_n^{m})$ and $(\theta_n^m)$ is a compatible left inverse of $(\psi_n^{n+1})$.
Assume $\D$ is a $\KK$-object and $\phi: \D \hookrightarrow \B$ is an $\LL\KK$-arrow with a left inverse $\pi: \B \surj \D$.  For given $\epsilon>0$, find $n$ and a unitary $u$ in $\widetilde\B$ such that $u^* \phi[\D] u \subseteq \psi_n^\infty [\B_n]$ and $\|u - 1\|<\epsilon/2$ (Lemma \ref{twisting}). Define $\psi: \D \hookrightarrow \B_n$ by $\psi(d) = \theta_n^\infty (u^*\phi(d) u)$. Then $\psi$ has a left inverse $\theta: \B_n \twoheadrightarrow \D$ defined by $\theta(x) =\pi(u\psi_n^\infty(x) u^*)$ (see the proof of Lemma \ref{sub-retract} (2)).
Condition $\|u - 1\|<\epsilon/2$ implies that $\| \psi_n^\infty(\psi(d)) -  \phi(d)\|<\epsilon$, for every $d$ in the unit ball of $\D$.
\end{proof}

\begin{lemma}
$\KK$ is separable.
\end{lemma}
\begin{proof}
There are, up to isomorphisms, countably many $\KK$-objects, namely finite sums of matrix algebras. The set of all embeddings between two fixed finite-dimensional $C^*$-algebras is a separable metric space. Thus, $\KK$ trivially has a countable dominating subcategory.
\end{proof}

The following statement is a direct consequence of Lemma~\ref{EPlim-retract}.

\begin{corollary}\label{CORlim-retract}
Suppose $\KK$ is a Fra\"\i ss\'e category of $C^*$-algebras with the Fra\"\i ss\'e limit $\mathcal U$ and $\ddagger \KK$ has the proper amalgamation property. Then $\mathcal U$ is a split-extension of every AF-algebra $\B$ in $\LL\KK$. In particular, $\mathcal U$ maps onto any AF-algebra in $\LL$.
\end{corollary}

\section{AF-algebras with  Cantor property as  Fra\"\i ss\'e limits}\label{Ax-sec} 
Suppose $\KK$ is a category of (not necessarily all) finite-dimensional $C^*$-algebras, closed under isomorphisms, and $\KK$-arrows are left-invertible embeddings.
\begin{definition}\label{closed-def}
We say $\KK$ is $\oplus$-\emph{stable} if it satisfies the following conditions.
\begin{enumerate}
\item If $\D$ is a $\KK$-object, then so is any retract (ideal) of $\D$,
\item $\D\oplus \E \in \KK$ whenever $\D, \E\in \KK$.
\end{enumerate}
\end{definition}

 In general $0$ is a retract of any $C^*$-algebra and therefore it is the initial object of any $\oplus$-stable category, unless, when working with the unital categories (when all the $\KK$-arrows are unital), which in that case $0$ is not a $\KK$-object anymore. Unital categories are briefly discussed in Section \ref{unital-sec}.

\begin{theorem}\label{closed-Fraisse}
 Suppose $\KK$ is a $\oplus$-stable category. Then $\ddagger \KK$ has the proper amalgamation property. In particular, $\KK$  is a  Fra\"\i ss\'e category.
\end{theorem}

\begin{proof}
 Suppose $\D, \E$ and $\F$ are $\KK$-objects and $\ddagger\KK$-arrows $(\phi, \pi):\D \to \E $ and $(\psi, \theta): \D \to \F$ are given. Since $\phi$ and $\psi$ are left-invertible, by Proposition \ref{fact1} we can identify $\E$ and $\F$ with $\E_0 \oplus\E_1$ and $\F_0 \oplus \F_1$, respectively, and find $\phi_0, \phi_1, \psi_0, \psi_1$ such that
 \begin{itemize}
 \item $\phi_0: \D \to \E_0$  and $\psi_0: \D \to \F_0$  are isomorphisms,
  \item $\phi_1: \D \to \E_1$ and   $\psi_1: \D \to \F_1$ are  homomorphisms,
  \item $\phi(d) = (\phi_0(d), \phi_1(d))$ and $\psi(d) = (\psi_0(d), \psi_1(d))$ for every $d\in \D$,
  \item $\pi(e_0, e_1) = \phi_0^{-1}(e_0)$ and $\theta(f_0, f_1) = \psi_0^{-1}(f_0)$.
 \end{itemize}
  
Define homomorphisms $\mu: \E \rightarrow \F_1$ and $\nu: \F \rightarrow \E_1$ by $\mu = \psi_1 \circ \pi$ and $\nu = \phi_1 \circ \theta$ (see Diagram (\ref{diag1})). Since $\KK$ is $\oplus$-stable $\D\oplus \E_1\oplus \F_1$ is a $\KK$-object. 
Define $\KK$-arrows $\phi': \E \hookrightarrow \D \oplus \E_1\oplus \F_1$ and $\psi': \F \hookrightarrow \D\oplus \E_1\oplus \F_1$ by 
$$\phi'(e_0,e_1) = (\phi_0^{-1}(e_0), e_1, \mu(e_0, e_1))$$
 and 
 $$\psi'(f_0, f_1) = (\psi_0^{-1}(f_0), \nu(f_0, f_1), f_1).$$
  For every $d\in \D$ we have 
$$
\phi'(\phi(d)) = \phi'(\phi_0(d), \phi_1(d))= (d, \phi_1(d), \mu(\phi(d))) = (d, \phi_1(d), \psi_1(d)) 
$$
and
$$
\psi'(\psi(d)) =\psi'(\psi_0(d), \psi_1(d)) =  (d, \nu(\phi(d)), \psi_1(d)) = (d, \phi_1(d), \psi_1(d)) .
$$

\begin{equation}\label{diag1}
\begin{tikzcd}[column sep=large]
 &  \mathcal \E \arrow[dr, hook',   "\phi'"] \arrow[dd, shift left=0.5ex, "\mu"] \arrow[dl, shift left=1ex, two heads, "\pi"]  \\ 
\D \arrow[ur, hook', "\phi"] \arrow[dr,  hook,  "\psi"] 
& & \D \oplus \E_1\oplus \F_1 \arrow[ul, shift left=1ex, two heads, "\pi'"] \arrow[dl,shift left=1ex, two heads, "\theta'"]\\
& \F \arrow[ur, hook, "\psi^\prime "]   \arrow[uu,shift left=0.5ex, "\nu"] \arrow[ul, shift left=1ex, two heads, "\theta"]
\end{tikzcd}
\end{equation}
Therefore $\phi^\prime \circ \phi = \psi^\prime \circ \psi $. The map $\pi^\prime: \D \oplus \E_1\oplus \F_1  \to \E$ defined by $\pi^\prime(d,e_1, f_1) = (\phi_0(d), e_1)$ is a left inverse of $\phi^\prime$. Similarly the map  $\theta^\prime: \D \oplus \E_1\oplus \F_1  \to \F$ defined by $\theta^\prime(d,e_1, f_1) = (\psi_0(d), f_1)$ is a left inverse of $\psi^\prime$. 
Therefore 
$(\phi^\prime, \pi^\prime): \E \to \D \oplus \E_1\oplus \F_1 $ and $(\psi^\prime, \theta^\prime): \E \to \D \oplus \E_1\oplus \F_1 $ are $\KK$-arrows.
We have 
$$\pi \circ \pi^\prime (d, e_1, f_1) = \pi (\phi_0(d), e_1) = d, $$
$$\theta \circ \theta^\prime (d, e_1, f_1) = \theta (\psi_0(d), e_1) = d. $$
Hence $\pi \circ \pi^\prime =\theta \circ \theta^\prime $. Also 

\begin{align*}
\theta^\prime \circ \phi^\prime(e_0, e_1) &= \theta^\prime  (\phi_0^{-1}(e_0), e_1, \mu(e_0, e_1))  
= (\psi_0(\phi_0^{-1}(e_0)), \mu(e_0,e_1)) \\
& =(\psi_0(\pi(e_0, e_1)), \psi_1(\pi(e_0,e_1))) = \psi (\pi(e_0,e_1)). 
\end{align*}

So $\theta^\prime \circ \phi^\prime = \psi \circ\pi$ and similarly we have $\phi \circ \theta  = \pi^\prime \circ \psi^\prime$. This shows that $\ddagger\KK$ has proper amalgamation property. Since $\KK$ is separable and has an initial object, in particular, it is  a  Fra\"\i ss\'e category.
\end{proof}

Therefore any $\oplus$-stable category $\KK$ has a unique  Fra\"\i ss\'e sequence; a $\KK$-sequence which satisfies (F).

\begin{notation}
Let $\A_\KK$ denote the Fra\"\i ss\'e limit of the $\oplus$-stable category $\KK$. 
\end{notation}

The AF-algebra $\A_\KK$ is $\KK$-universal and almost $\KK$-homogeneous, by Theorem~\ref{uni-hom} and Lemma~\ref{L1-L3}. In fact, $\A_\KK$ is $\KK$-homogeneous (where $\epsilon$ is zero). To see this, suppose $\F$ is a finite-dimensional $C^*$-algebra in $\KK$ and $\phi_i: \F \hookr \A$ $(i=0,1)$ are left-invertible embeddings. By the almost $\KK$-homogeneity, there is an automorphism $\eta: \A_\KK \to \A_\KK$ such that $\|\eta \circ \phi_0 - \phi_1\|< 1$. There exists (Lemma \ref{fd-af}) a unitary $u\in\widetilde\A$ such that $\Ad_u \circ \eta \circ \phi_0 = \phi_1$. The automorphism $\Ad_u \circ \eta$ witnesses the $\KK$-homogeneity of $\A_\KK$.

Moreover, since $\ddagger \KK$ has the proper amalgamation property, every AF-algebra in $\LL\KK$, is a retract of $\A_\KK$ (Corollary \ref{CORlim-retract}).

\begin{corollary}\label{inj. uni. AF} Suppose $\KK$ is a $\oplus$-stable category, then 
\begin{itemize}
\item (universality) Every AF-algebra which is the limit of a $\KK$-sequence, is a retract of $\A_\KK$.
\item ($\KK$-homogeneity) For every finite-dimensional $C^*$-algebra $\F\in \KK$ and  left-invertible embeddings $\phi_i: \F \hookr \A_\KK$ ($i=0,1$), there is an automorphism $\eta: \A_\KK \to \A_\KK$ such that $\eta \circ \phi_0 = \phi_1$.
\end{itemize}
\end{corollary}

We will describe the structure of $\A_\KK$ by showing that it has the Cantor property.
\begin{lemma}\label{A_K-CS}
Suppose $\KK$ is a $\oplus$-stable category, then $\A_\KK$ has the Cantor property.
\end{lemma}

\begin{proof} 
  Suppose $\A_\KK = \varinjlim_n (\A_n, \phi_n^m)$, where $(\A_n, \phi_n^m)$ is a $\KK$-sequence, i.e., $(\A_n, \phi_n^m)$ is a left-invertible sequence of finite-dimensional $C^*$-algebras in $\KK$. Since $\A_\KK$ is the Fra\"\i ss\'e limit of $\KK$, we can suppose $(\A_n, \phi_n^m)$ satisfies (F). We claim that $(\A_n, \phi_n^m)$ satisfies (D0)--(D2) of Definition \ref{Bratteli}.  
  Suppose $\mathfrak D$ is the Bratteli diagram of $(\A_n, \phi_n^m)$ and $\A_n = \A_{n,1} \oplus \dots \oplus \A_{n,k_n}$ for every $n$, such that each $\A_{n,s}$ is a matrix algebra. 
  
  The condition (D0) is trivial since $\phi_n^m$ are left-invertible. 
To see (D1), fix $\A_{n,s}$. Note that since $\A_n$ is a $\KK$-object and $\KK$ is $\oplus$-stable, we have $\A_{n} \oplus \A_{n} \in \KK$.  Let $\gamma: \A_{n} \hookr \A_{n} \oplus \A_{n}$ be the left-invertible embedding defined by $\gamma(a) = (a,a)$.
 Use the Fra\"\i ss\'e condition (F) to find $\delta: \A_n \oplus \A_{n} \hookr \A_{m}$, for some $m\geq n$, such that $\delta \circ \gamma = \phi_n^m$. 
Since $\delta$ is left-invertible, there are distinct $(m,t)$ and $(m,t')$ in $\mathfrak D$ such that $\A_{m,t} \cong \A_{m,t'} \cong \A_{n,s}$. Then $\delta \circ \gamma  = \phi_n^m$ implies that $(n,s)\to (m,t)$ and $(n,s) \to (m,t')$ in $\mathfrak D$.

To see (D2)  assume $\D \subseteq \A_n$ is an ideal of $\A_n$ and $M_\ell$ is a retract of $\A_\KK$ and there is an embedding $\gamma:\D \hookr M_\ell$. Suppose $\A_n= \D \oplus \E$ for some $\E$. Since $\KK$ is $\oplus$-stable, $\D \oplus \E \oplus M_\ell$ is a $\KK$-object.  Therefore $\gamma': \D\oplus \E  \hookr \D \oplus \E  \oplus  M_\ell$ defined by $\gamma'(d,e) = (d, e, \gamma(d))$ is a  $\KK$-arrow. Then by $(F)$ there is a left-invertible embedding $\delta': \D \oplus \E  \oplus M_\ell \hookr \A_m$ for some $m\geq n$, such that 
\begin{equation}\label{eq-delta}
\delta'\circ \gamma' = \phi_n^m.    
\end{equation}
Since $\delta'$ is left-invertible, there is $(m,t)$ such that $\dim(\A_{m,t}) = \ell$ and $$\delta_{m,t}=\pi_{\A_{m,t}} \circ \delta|_{M_\ell} : M_\ell \hookr \A_{m,t}$$
is an isomorphism, where $\pi_{\A_{m,t}}: \A_m \surj \A_{m,t}$ is the canonical projection.
 Let 
 $$\phi_{m,t}=\pi_{\A_{m,t}} \circ \phi_n^m|_\D : \D \to \A_{m,t}.$$
By definition of $\gamma'$ and (\ref{eq-delta}) it is clear that $\phi_{m,t} = \delta_{m,t}\circ \gamma$ and that $\phi_{m,t}$ is also an embedding.  By Lemma \ref{matrix-absorbing} we have $ \Mult_{\phi_{m,t}}(\D, \A_{m,t}) = c \Mult_\gamma(\D , M_\ell ) $ for some natural number $c\geq 1$. Since $\delta_{m,t}$ is an isomorphism, we have $c=1$. This proves (D2).
\end{proof}

Next we show that every AF-algebra with  Cantor property  can be realized as the  Fra\"\i ss\'e limit of a suitable $\oplus$-stable category of finite-dimensional $C^*$-algebras and left-invertible embeddings. 
\subsection{The category $\KK_\A$} 

Suppose $\A$ is an AF-algebra with  Cantor property. Let $\KK_\A$ denote the category whose objects are finite-dimensional retracts of $\A$ and $\KK_\A$-arrows are left-invertible embeddings.
Let $\LL_\A$ be the category whose objects  are limits of $\KK_\A$-sequences. If $\B$ and $\mathcal C$ are $\LL_\A$-objects, an $\LL_\A$-arrow from $\B$ into $\mathcal C$ is a left-invertible embedding $\phi: \B \hookrightarrow \mathcal C$.

\begin{lemma}\label{K_A closed}
 $\KK_\A$ is a  Fra\"\i ss\'e category and  $\ddagger \KK_\A$ has the proper amalgamation property.
\end{lemma}
\begin{proof}
By Theorem \ref{closed-Fraisse}, it is enough to show that $\KK_\A$  is a $\oplus$-stable category.
 Condition (1) of Definition \ref{closed-def} is trivial. Condition (2) follows from Proposition \ref{sum-retract}.
\end{proof}

Again, Theorem \ref{uni-hom} guarantees the existence of a unique $\KK_\A$-universal and $\KK_\A$-homogeneous  AF-algebra in $\LL_\A$, namely the Fra\"\i ss\'e limit of $\KK_\A$.

\begin{theorem}\label{Cantor-closed-Fraisse}
The Fra\"\i ss\'e limit of $\KK_\A$ is $\A$.
\end{theorem}

\begin{proof}  There is a sequence $(\A_n, \phi_n^m)$  of finite-dimensional $C^*$-algebras and embeddings such that $\A = \varinjlim (\A_n, \phi_n^m)$ satisfies (D0)--(D2) of Definition \ref{Bratteli}. First note that by (D0), $(\A_n, \phi_n^m)$ is a $\KK_\A$-sequence and therefore $\A$ is an $\LL_\A$-object.
In order to show that $\A$ is the Fra\"\i ss\'e limit of  $\KK_\A$, we need to show that $(\A_n, \phi_n^m)$  satisfies condition (F). This is Lemma \ref{absorbing}.
\end{proof}

\begin{theorem}\label{A_K-characterization}
Suppose $\KK$ is a $\oplus$-stable category. $\A_\KK$ is the unique AF-algebra such that 
\begin{enumerate}
    \item it has the Cantor property,
    \item a finite-dimensional $C^*$-algebra is a retract of $\A_\KK$ if and only if it is a $\KK$-object.
\end{enumerate}
\end{theorem}
\begin{proof}
We have already shown that $\A_\KK$ has the Cantor property  (Lemma \ref{A_K-CS}). By Lemma \ref{sub-retract}(2), every finite-dimensional retract of $\A_\KK$ is a $\KK$-object and every finite-dimensional $C^*$-algebra in $\KK$ is a retract of $\A_\KK$, by the $\KK$-universality of $\A_\KK$. If $\A$ is an AF-algebra satisfying (1) and (2), then by definition $\KK_\A= \KK$. The uniqueness of the Fra\"\i ss\'e limit and Theorem \ref{Cantor-closed-Fraisse} imply that $\A \cong \A_\KK$.   
\end{proof}
\begin{corollary}\label{retract-coro}
 Two AF-algebras with Cantor property are isomorphic if and only if they have the same set of matrix algebras as retracts. 
\end{corollary}

\subsection{Examples}\label{tensor-cp}

Corollary \ref{retract-coro} shows that there is a one to one correspondence between AF-algebras with the Cantor property and the collections of (non-isomorphic) matrix algebras (hence, with the subsets of the natural numbers). More precisely, given any collection $X$ of non-isomorphic matrix algebras, let $\KK_X$ denote the $\oplus$-stable category whose objects are finite direct sums of the matrix algebras in $\KK_X$ (finite direct sums of a member of $X$ with itself are of course allowed) and left-invertible embeddings as arrows. Then the Fra\"\i ss\'e limit of $\KK_X$ is the unique AF-algebra whose  matrix algebra retracts are exactly the members of $X$.

The class of AF-algebras with the Cantor property is not closed under direct sum (for instance, $(M_2 \oplus M_3) \otimes C(2^\mathbb N)$ does not have the Cantor property, as its Bratteli diagram easily reveals, while $M_2\otimes C(2^\mathbb N)$ and $M_3 \otimes C(2^\mathbb N)$ do). The following example shows that this class is also not closed under tensor product.

Let $\A$ denote the unique AF-algebra with the Cantor property whose matrix algebra retracts are exactly $\{M_2, M_3, M_5, M_{11}\}$. We claim that $\A \otimes \A$ does not have the Cantor property.
Suppose $\A = \varinjlim (\A_n, \phi_n^m)$  where the sequence satisfies (D0)--(D2) of Definition \ref{Bratteli}. 
 Clearly  $\A\otimes \A$ is the limit of the left-invertible sequence $ (\A_n\otimes \A_n, \phi_n^m\otimes \phi_n^m)$. Therefore by Lemma \ref{sub-retract} every matrix algebra retract of  $\A\otimes \A$ is isomorphic to $\D \otimes \E$, where $\D, \E\in \{M_2, M_3, M_5 , M_{11}\}$. 
Take a retract of $\A_n\otimes \A_n$ isomorphic to $M_3 \otimes M_5$ (for large enough $n$ there is such a retract) and let $\gamma : M_3 \otimes M_5 \to M_2 \otimes M_{11}$ be an embedding of multiplicity $1$. However, there is no embedding $\phi\otimes \psi :  M_3 \otimes M_5 \to M_{22} \cong M_2 \otimes M_{11}$ which corresponds to a path in the Bratteli diagram of the sequence  $ (\A_n\otimes \A_n, \phi_n^m\otimes \phi_n^m)$. This is because the codomain of any such $\phi$ should be either $M_3$ or $M_5$ or $M_{11}$ (since $\phi$ corresponds to a path in the  Bratteli diagram of the sequence  $ (\A_n, \phi_n^m)$) and similarly the codomain of $\psi$ could only be $M_5$ or $M_{11}$, while the tensor product of their codomains should be isomorphic to $M_{22}$, which is not possible.  Thus condition (D2) is satisfied neither by the sequence  $ (\A_n\otimes \A_n, \phi_n^m\otimes \phi_n^m)$,   nor by any sequence of finite-dimensional $C^*$-algebras whose limit is $\A \otimes \A$ (see Remark \ref{remark-CP-def}), which means that $\A\otimes \A$ does not have the Cantor property.

\section{Universal AF-algebras}\label{universal-section}

Let ${\mathfrak F}$ denote the category of \emph{all} finite-dimensional $C^*$-algebras and left-invertible embeddings. 
 The category ${\mathfrak F}$ is $\oplus$-stable and therefore it is Fra\"\i ss\'e  by Theorem \ref{closed-Fraisse}. The Fra\"\i ss\'e limit $\A_\mathfrak F$ of this category has the universality property (Corollary \ref{inj. uni. AF}) that any AF-algebra which is the limit of a left-invertible sequence of finite-dimensional $C^*$-algebras can be embedded via a left-invertible embedding into $\A_\mathfrak F$. In fact,  $\A_\mathfrak F$ is surjectively universal in the category of all (separable) AF-algebras. 
 \begin{theorem}\label{universal AF-algebra}
There is a surjective homomorphism from $\A_{\mathfrak F}$ onto any separable AF-algebra.
\end{theorem}
\begin{proof}
Suppose $\B$ is a separable AF-algebra.
Proposition \ref{ess-quotient} states that there is an AF-algebra $\A$, which is the limit of a left-invertible sequence of finite-dimensional $C^*$-algebras  and $\A/\J \cong \B$, for some ideal $\J$.  
 By the universality of $\A_{\mathfrak F}$ (Corollary \ref{inj. uni. AF}) there is a left-invertible embedding $\phi: \A \hookr \A_{\mathfrak F}$. If $\theta: \A_{\mathfrak F} \surj \A$ is a left inverse of $\phi$ then its composition with the quotient map $\pi: \A \surj \A / \J$ gives a  surjective homomorphism from $\A_{\mathfrak F}$ onto $\B$.
\end{proof}
 
\begin{remark}\label{not unique}
 Since $\A_\mathfrak F$ has the Cantor property (Lemma \ref{A_K-CS}), it does not have any minimal projections. Therefore, for example, it cannot be isomorphic to $\A_\mathfrak F \oplus \mathbb C$. Hence the property of being surjectively universal AF-algebra is not unique to $\A_\mathfrak F$. 
\end{remark}

 \begin{corollary}\label{main-corollary}
  An AF-algebra $\mathcal A$ is surjectively universal if and only if  $\A_\mathfrak F$ is a quotient of $\A$.
 \end{corollary}
 
 Theorem \ref{A_K-characterization} provides a characterization of $\A_\mathfrak F$, up to isomorphism, in terms of its structure.

\begin{corollary}\label{Bratteli diagram-A_F}
$\A_\mathfrak F$ is the unique separable AF-algebra with  Cantor property such that every matrix algebra $M_k$ is a retract of $\A$.

Equivalently, an AF-algebra $\A$ is isomorphic to $\A_\mathfrak F$ if and only if there is a sequence $(\A_n, \phi_n^m)$  of finite-dimensional $C^*$-algebras and embeddings such that $\A = \varinjlim (\A_n, \phi_n^m)$ and the Bratteli diagram  $\mathfrak D$ of $(\A_n, \phi_n^m)$  satisfies (D0)-(D2) and 
\begin{enumerate}
\item[(D3)] for every $k$  there is $(n,s)\in \mathfrak D$ such that $\dim(n,s) = k$. 
\end{enumerate}
\end{corollary}

\begin{theorem}\label{A_F-char}
$\A_\mathfrak F$ is the unique AF-algebra that is the limit of a left-invertible sequence of finite-dimensional $C^*$-algebras and for any finite-dimensional $C^*$-algebras $\D, \E$ and left-invertible embeddings  $\phi: \D \hookr \E$ and $\alpha: \D \hookr \A_\mathfrak F$ there is a left-invertible embedding $\beta:\E \hookr \A_\mathfrak F$ such that $\beta \circ \phi = \alpha$. 
\end{theorem}

\begin{proof}
 Suppose $\A_\mathfrak F$ is the limit of the Fra\"\i ss\'e  $\mathfrak F$-sequence $(\A_n, \phi_n^m)$. By definition, $\alpha$ and $\phi$ are $\mathfrak F$-arrows. There is (Lemma~\ref{L1-L3}) a natural number $n$ and an $\mathfrak F$-arrow (a left-invertible embedding) $\psi: \D \hookr \A_n$ such that $\|\phi_n^\infty \circ \psi - \alpha\| <1$. Use the amalgamation property to find a finite-dimensional $C^*$-algebra $\G$ and left-invertible embeddings $\phi' : \E \hookr \G$ and $\psi': \A_n \hookr \G$ such that $\phi' \circ \phi = \psi' \circ \psi$ (see Diagram (\ref{G2-A})). The Fra\"\i ss\'e condition (F) implies the existence of $m\geq n$ and a left-invertible embedding $\delta : \G \hookr \A_m$ such that $\delta \circ \psi'= \phi_n^m$. Let $\beta' = \phi_m^\infty \circ \delta \circ \phi'$. It is clearly left-invertible. 

\begin{equation}\label{G2-A}
\begin{tikzcd}[row sep=small]
\A_1 \arrow[r, hook, "\phi_1^2"] 
& \A_2   \arrow[r, hook,  "\phi_2^3"]  & \dots   \arrow[r, hook] &\A_n 
\arrow[dr, hook,  "\psi'"]  \arrow[rr, hook,  "\phi_n^m"]&  & \A_m \arrow[rr, hook,  "\phi_m^\infty"]&  &\A_\mathfrak F\\ 
&&& & \G \arrow[ur, hook,  "\delta"]& \\
&&  \D \arrow[uur, hook',  "\psi"]   \arrow[rr, hook,  "\phi"] \arrow[uurrrrr, hook, bend right=40, crossing over, "\alpha"]  &  & \E \arrow[u, hook,  "\phi'"] \arrow[uurrr, hook, dashed, "\beta"]&
\end{tikzcd}
\end{equation}
 For every $d$ in  $\D$ we have 
$$ 
\beta' \circ \phi (d) = \phi_m^\infty \circ \delta \circ \phi' \circ \phi (d) =  \phi_m^\infty \circ \delta \circ \psi' \circ \psi (d)
=  \phi_m^\infty \circ \phi_n^m\circ \psi (d) =  \phi_n^\infty \circ \psi (d).
$$
Therefore $\|\beta' \circ \phi  - \alpha\|<1$. Conjugating $\beta'$ with a unitary in $\widetilde \A_\mathfrak F$ gives the required left-invertible embedding $\beta$ (Lemma \ref{fd-af}). 

For the uniqueness,  suppose $\B$ is the limit of a left-invertible sequence $(\B_n, \psi_n^m)$ of finite-dimensional $C^*$-algebras, satisfying the assumption of the theorem. Using this assumption we can show that $(\B_n, \psi_n^m)$ satisfies the Fra\"\i ss\'e condition (F) and therefore $\B$ is the Fra\"\i ss\'e limit of $\mathfrak F$. Uniqueness of the Fra\"\i ss\'e limit implies that $\B$ is isomorphic to $\A_\mathfrak F$.
\end{proof}

Let us conclude this section with another example of a Fra\"\i ss\'e category of finite-dimensional $C^*$-algebras.

\begin{remark} Note that a similar argument as in \ref{tensor-cp} shows that $\A_{\mathfrak F} \otimes \A_{\mathfrak F}$ does not have the Cantor property. In particular, $\A_{\mathfrak F}$ is not self-absorbing, i.e., $\A_{\mathfrak F} \otimes \A_{\mathfrak F}$ is not isomorphic to $\A_{\mathfrak F}$. 
\end{remark}

\subsection{The universal UHF-algebra}\label{universal uhf} Recall that a UHF-algebra is the (inductive) limit of
$$M_{k_1} \xrightarrow{\phi_1^2} M_{k_2} \xrightarrow{\phi_2^3} M_{k_3} \xrightarrow{\phi_3^4} \dots $$
of full matrix algebras, with unital connecting maps $\phi_n^{n+1}$. In particular $k_j|k_{j+1}$ for each $j$. To each sequence of natural numbers $\{k_j\}_{j\in \mathbb N}$ (hence to the corresponding UHF-algebra) a \emph{supernatural number} $n$ is associated, which is the formal product 
$$n = \prod_{p \text{ prime}} p^{n_p}$$
where  $n_p\in \{0,1,2, \dots, \infty\}$ and for each prime number $p$, 
$$n_p =\sup\{n: p^n|k_j \text{ for some } j\}.$$
Also to each supernatural number $n$ there is an associated UHF-algebra denoted, as it is common, by $M_n$ (e.g., the CAR-algebra is $M_{2^\infty}$). Glimm \cite{Glimm} showed that a supernatural number is a complete invariant for the associated UHF-algebra. Recall that   the \emph{universal UHF-algebra} (see \cite{Rordam-Stormer}), denoted by $\mathcal Q$, is the UHF-algebra associated to the supernatural number 
 $$n_\infty = \prod_{p \text{ prime}} p^{\infty}.$$
 The universal UHF-algebra $\Q$ is also the unique unital AF-algebra such that 
 $$\la K_0(\Q), K_0(\Q)_+,  [1_\Q] \ra \cong \la \mathbb Q, \mathbb Q_+, 1 \ra.$$
The multiplication of supernatural numbers is defined in the obvious way which means for supernatural numbers $n,m$ we have $M_n \otimes M_m \cong M_{nm}$. This in particular implies that $\Q \otimes \mathcal M \cong \Q$, for any UHF-algebra $\mathcal M$.

 Now suppose $\mathfrak M$ is the category of all nonzero matrix algebras and unital embeddings. Then $\mathfrak M$ is a Fra\"\i ss\'e category. The only nontrivial part of the latter statement is to show that $\mathfrak M$ has the amalgamation property, but this is quite easy since it is enough to make sure that the composition maps have the same multiplicities and then conjugating with a unitary makes sure that the composition maps are the same (this is similar to the proof of the amalgamation property in \cite[Theorem 3.4]{Eagle}). The Fra\"\i ss\'e limit of $\mathfrak M$ is $\mathcal Q$, since the universality property of the Fra\"\i ss\'e limit implies that the supernatural number associated to it must be $n_\infty$.

\section{Unital categories}\label{unital-sec}
The proof of Theorem \ref{closed-Fraisse} also shows that the category of all finite-dimensional $C^*$-algebras (or any $\oplus$-stable category) and \emph{unital} left-invertible embeddings  has the (proper) amalgamation property. However, this category fails to have the joint embedding property (note that $0$ is no longer an object of the category), since for example one cannot jointly embed $M_2$ and $M_3$ into a finite-dimensional $C^*$-algebra with unital left-invertible maps.

\subsection{The category $\widetilde{\mathfrak F}$}
Let $\widetilde{\mathfrak F}$ denote the category of all finite-dimensional $C^*$-algebras isomorphic to $\mathbb C \oplus \D$, for a finite-dimensional $C^*$-algebra $\D$, and unital left-invertible embeddings. This category is no longer $\oplus$-stable, however, a similar proof to the one of Theorem \ref{closed-Fraisse}, where the maps are unital, shows that $\ddagger \widetilde{\mathfrak F}$ has the proper amalgamation property. Therefore $\widetilde{\mathfrak F}$ is a  Fra\"\i ss\'e category, since $\mathbb C$ is the initial object of this category and therefore the joint embedding property is a consequence of the amalgamation property. The Fra\"\i ss\'e  limit $\A_{\widetilde{\mathfrak F}}$ of this category is a separable AF-algebra with the universality property that any unital AF-algebra which can be obtained as the limit of a  left-invertible unital sequence of finite-dimensional  $C^*$-algebras isomorphic to $\mathbb C \oplus \D$, can be embedded via a  left-invertible unital embedding into $\A_{\widetilde{\mathfrak F}}$. The unital analogue of Theorem \ref{universal AF-algebra} states the following. 

\begin{corollary}
 For every unital separable AF-algebra $\B$ there is a surjective  homomorphism from $\A_{\widetilde{\mathfrak F}}$ onto $\B$.
\end{corollary}
\begin{proof}
Suppose $\B$ is an arbitrary unital AF-algebra. Using Proposition \ref{ess-quotient} we can find a unital AF-algebra $\A\supseteq B$ which is the limit  of a  left-invertible unital sequence of finite-dimensional $C^*$-algebras, such that $\B$ is a quotient of $\A$. Thus $\mathbb C \oplus \A$ is the limit of a unital left-invertible sequence of finite-dimensional $C^*$-algebras of the form $\mathbb C \oplus \D$, for finite-dimensional $\D$. By the universality of $\A_{\widetilde{\mathfrak F}}$, there is a left-invertible unital embedding from $\mathbb C \oplus \A$ into $\A_{\widetilde{\mathfrak F}}$. Since $\B$ is a quotient of $\A$, there is a surjective homomorphism from $\mathbb C \oplus \A$ onto $\B$. Combining the two surjections gives us a surjective homomorphism from $\A_{\widetilde{\mathfrak F}}$ onto $\B$.
\end{proof}
\begin{remark}
Small adjustments in the proof of Lemma \ref{A_K-CS} show that $\A_{\widetilde{\mathfrak F}}$ has the Cantor property (in the sense of Definition \ref{CS-unital-def}).
In fact, it is easy to check that $\A_{\widetilde{\mathfrak F}}$ is isomorphic to  $\widetilde{\A_\mathfrak F}$, the unitization of $\A_\mathfrak F$.  This, in particular, implies that $\A_\mathfrak F$ is not unital. Since if it was unital, then $\widetilde{\A_\mathfrak F}$ (and hence $\A_{\widetilde{\mathfrak F}}$) would be isomorphic to $\A_\mathfrak F \oplus \mathbb C$, but this is not possible since $\A_{\widetilde{\mathfrak F}}$  has the Cantor property and therefore has no minimal projections.
\end{remark}


\begin{definition}
We say $\D$ is a \emph{unital-retract} of the $C^*$-algebra $\A$ if there is a  left-invertible unital embedding from $\D$ into $\A$.
\end{definition}

\subsection{The category $\widetilde{\KK}_\A$}
If $\A$ is a unital AF-algebra with  Cantor property (Definition \ref{CS-unital-def}), then let $\widetilde{\KK}_\A$ denote the category whose objects are finite-dimensional unital-retracts of $\A$ and  morphisms are unital left-invertible embeddings. This category is not $\oplus$-stable, since it does not satisfy condition (1) of Definition \ref{closed-def}.  However, $\ddagger\widetilde{\KK}_\A$ still has the proper amalgamations property.

\begin{proposition}
$\ddagger\widetilde{\KK}_\A$ has the proper amalgamation property.
\end{proposition}
\begin{proof}
The proof is exactly the same as the proof of Lemma \ref{closed-Fraisse} where the maps are assumed to be unital. We only need to check that $\D \oplus \E_1 \oplus \F_1$ is a unital-retract of $\A$. By Lemma \ref{sub-retract}, for some $m$ both $\E\cong \D \oplus \E_1$ and $\F \cong \D \oplus \F_1$ are unital-retracts of $\A_m$.  An easy argument using Proposition \ref{fact1} shows that $\D \oplus \E_1 \oplus \F_1$ is also a unital-retract of $\A_m$ and therefore a unital retract of $\A$.
\end{proof}
Also $\widetilde{\KK}_\A$ has a weakly initial object (by the next lemma). Therefore it is a Fra\"\i ss\'e category. Recall that an object is \emph{weakly initial} in $\KK$ if it has at least one $\KK$-arrow to any other object of $\KK$.

\begin{lemma}\label{initial}
Suppose $\A$ is a unital AF-algebra with  Cantor property. The category $\widetilde{\KK}_\A$ has a weakly initial object, i.e., there is  a finite-dimensional unital-retract of $\A$ which can be mapped into any other finite-dimensional unital-retract of $\A$ via a left-invertible unital embedding.
\end{lemma}
\begin{proof}
Let $M_{k_1} \oplus \dots \oplus M_{k_l}$ be an arbitrary  $\widetilde{\KK}_\A$-object. Suppose that $\{k'_1, \dots, k'_t\}$ is the largest subset of $\{k_1, \dots, k_l\}$ such that $k'_i$ cannot be written as $\sum_{\substack{j \leq n\\ j\neq i}} x_j k'_j$ for any natural set numbers $\{x_j: j \leq n \text{ and } j\neq i\}$,  for any $i\leq t$. Since $\{k'_1, \dots, k'_t\}$ is the largest such subset, $\D= M_{k'_1} \oplus \dots \oplus M_{k'_t}$ is a  unital-retract of $M_{k_1} \oplus \dots \oplus M_{k_l}$ and therefore a  unital-retract of $\A$. Suppose $\F$ is an arbitrary $\widetilde{\KK}_\A$-object. 
 Let $(\A_n, \phi_n^m)$ be a $\widetilde{\KK}_\A$-sequence  with limit $\A$ such that $\A_1\cong \F$. Then $\D$ is a unital-retract of some $\A_m$, so $\A_m = \dot \D \oplus \E$, for some $\E$ and $\dot \D \cong \D$. 
 
 Fix $i\leq t$. Since $\phi_1^m$ is a unital embedding, there is a subalgebra of $\F$ isomorphic to $M_{n_1} \oplus \dots \oplus M_{n_s}$ such that $\sum_{j=1}^s y_j n_j = k'_i$, for some $\{y_1, \dots, y_s\}\subseteq \mathbb N$.  We claim that exactly one $n_j$ is equal to $k'_i$ and the rest are zero. If not, then for every $j\leq s$ we have $0<n_j < k'_i$. Since $\phi_1^m$ is left-invertible, for every $j\leq s$ a copy of $M_{n_j}$ appears as a summand of $\A_m$. Also because there is a unital embedding from $\D$ into $\A_m$, for some $\{x_1, \dots, x_r\}\subseteq \mathbb N$ we have $n_j = \sum_{\substack{j' \leq r\\ j'\neq i}} x_{j'} k'_{j'}$ for every $j\leq s$. But then 
 $$k'_i = \sum_{j=1}^s \sum_{\substack{j' \leq n\\ j'\neq i}} x_{j'} y_j k'_{j'}, $$ 
 which is a contradiction with the choice of $k_i'$.  This means that $\F =  \F_0 \oplus \F_1$ such that $ \F_0 \cong \D$ and there is a unital homomorphism from $\D$ onto $\dot \F_1$. Therefore $\D$ is a unital-retract of $\F$.
\end{proof}

\begin{corollary}
Suppose $\A$ is a unital AF-algebra with  Cantor property.
The category $\widetilde{\KK}_\A$ is a Fra\"\i ss\'e category and $\ddagger\widetilde{\KK}_\A$ has the proper amalgamation property. The  Fra\"\i ss\'e limit of  $\widetilde{\KK}_\A$ is $\A$.
\end{corollary}
\begin{proof}
The proof of the fact that $\A$ is the Fra\"\i ss\'e limit of $\widetilde{\KK}_\A$ is same as Theorem \ref{Cantor-closed-Fraisse}, where all the maps are unital.  
\end{proof}

\section{Surjectively universal countable dimension groups}\label{K0-section}

A countable partially ordered abelian group $\la G, G^+ \ra$ is a (countable) \textit{dimension group} if it is isomorphic to the inductive limit of a sequence
$$\mathbb Z^{r_1} \xrightarrow{\alpha_1^2}\mathbb Z^{r_2} \xrightarrow{\alpha_2^3}\mathbb Z^{r_3} \xrightarrow{\alpha_3^4} \dots $$
for some natural numbers $r_n$, where $\alpha_i^j$ are positive group homomorphisms and $\mathbb Z^r$ is equipped with the ordering given by
$$(\mathbb Z^r)^+ = \{(x_1, x_2, \dots,x_r)\in \mathbb Z^r: x_i\geq 0 \text{ for } i=1,\dots,r\}.$$
A partially ordered abelian group that
is isomorphic to $\la \mathbb Z^r , (\mathbb Z^r)^+ \ra $, for a non-negative integer $r$, is usually called a \emph{simplicial group}.
A \emph{scale} $S$ on the dimension group $\la G, G^+ \ra$ is a generating, upward directed and hereditary  subset of $G^+$ (see \cite[IV.3]{Davidson}).

\begin{notation}\label{Notation}
If  $\la G, S \ra$ is a scaled dimension group as above,  we can recursively pick order-units  
$$\bar u_n = (u_{n,1}, u_{n,2}, \dots, u_{n,r_n})\in (\mathbb Z^{r_n})^+$$
of $\mathbb Z^{r_n}$ such that $\alpha_n^{n+1}(\bar u_n)\leq \bar u_{n+1} $ and $ S = \bigcup_n \alpha_n^\infty [[\bar 0, \bar u_n]]$.
Then we say the scaled dimension group $\la G, S\ra$  is the limit of the sequence $( \mathbb Z^{r_n}, \bar u_n, \alpha_n^m )$. If $(\bar u_n)$ can be chosen such that  $\alpha_n^{n+1} (\bar u_n) = \bar u_{n+1}$ for every $n\in \mathbb N$, then $G$ has an order-unit $u = \lim_n \alpha_n^\infty(\bar u_n)$. In this case we denote this dimension group with order-unit by $\la G, u\ra$.
\end{notation}

An isomorphism between scaled dimension groups is a positive group isomorphism which sends the scale of the domain to the scale of the codomain.
Given  a separable AF-algebra $\A$, its $K_0$-group $\la K_0(\A), K_0(\A)^+\ra$ is a (countable) dimension group and conversely any dimension group is isomorphic to   $K_0$-group of a separable AF-algebra. 
The \emph{dimension range} of $\A$,
 $$\mathcal D(\A) = \{[p]: p \text{ is a projection of }\A\}\subseteq K_0(\A)^+$$ 
 is a scale for $\la K_0(\A), K_0(\A)^+ \ra$, and therefore  $\la K_0(\A), \D(\A)\ra$ is a scaled dimension group. Conversely, every scaled dimension group is isomorphic to  
$\la K_0(\A), \D(\A)\ra$ for a separable AF-algebra $\A$.
 Elliott's classification of separable AF-algebras (\cite{Elliott})  states that $\la K_0(\A), \D(\A)\ra$ is a complete isomorphism invariant for the separable AF-algebra $\A$. 

\begin{theorem}[Elliott \cite{Elliott}]
Two separable AF-algebras $\A$ and $\B$ are isomorphic if and only if their scaled dimension groups are isomorphic. If $\A$ and $\B$ are unital, then they are isomorphic if and only if  $\la K_0(\A), [1_\A]\ra \cong \la K_0(\B), [1_\B]\ra$, as partially ordered abelian groups with order-units. 
\end{theorem}

\subsection{Surjectively universal dimension groups}
The universality property of $\la K_0(\A_\mathfrak F), \D(\A_{\mathfrak F}) \ra$ can be obtained by applying $K_0$-functor to Theorem \ref{universal AF-algebra}.

\begin{corollary}
The scaled (countable) dimension group $\la K_0(\A_\mathfrak F), \D(\A_{\mathfrak F}) \ra$  maps onto any countable scaled dimension group.
\end{corollary}

By applying $K_0$-functor to Corollary \ref{Bratteli diagram-A_F}, we immediately obtain the following result.
\begin{corollary}\label{coro-uni-K0}
$\la K_0(\A_\mathfrak F), \D(\A_{\mathfrak F}) \ra$ is the unique scaled dimension group which is the limit of a sequence $( \mathbb Z^{r_n}, \bar u_n, \alpha_n^m )$ (as in Notation above) satisfying the following conditions:
\begin{enumerate}
\item for every $n\in \mathbb N$ and $1 \leq i \leq r_n$ there are $m\geq n $ and $1\leq j , j^\prime \leq r_m$ such that $j\neq j^\prime$, 
$u_{n,i} = u_{m,j} = u_{m,j^\prime}$ and  $\pi_j \circ \alpha_n^m(u_{n,i}) = u_{m,j}$ and $\pi_{j^\prime} \circ \alpha_n^m(u_{n,i}) = u_{m,j^\prime}$, where $\pi_j$ is the canonical projection from $\mathbb Z^{r_m}$ onto its $j$-th coordinate.

\item  for every $n,n^\prime\in \mathbb N$,  $1\leq  i^\prime \leq r_{n^\prime}$ and $\{x_1, \dots, x_{r_n}\}\subseteq \mathbb N \cup \{0\}$ such that 
$\sum_{i=1}^{r_n} x_i u_{n,i} \leq u_{n^\prime, i^\prime} $ there are $m\geq n$ and $1\leq j\leq r_m$ such that 
$u_{n^\prime,i^\prime} \leq u_{m,j}$ and $\pi_j \circ \alpha_n^m(u_{n,i}) = x_i.u_{n,i}$ for every $i \in \{1, \dots, r_n\}$.

\item For every $k\in \mathbb N$ there are natural numbers $n$ and $1\leq i \leq r_n$ such that $u_{n,i} = k$.
\end{enumerate}
\end{corollary}

\begin{corollary}
The (countable) dimension group  with order-unit $\la  K_0(\A_{\widetilde{\mathfrak F}}), [1_{\A_{\widetilde{\mathfrak F}}}] \ra$   maps onto (there is a surjective normalized positive group homomorphism) any countable dimension group with order-unit.
\end{corollary}

A similar characterization of the dimension group with order-unit  $\la K_0(\A_{\widetilde{\mathfrak F}}), [1_{\A_{\widetilde{\mathfrak F}}}] \ra$ holds where $\alpha_n^m$ are order-unit preserving and in condition (2) of Corollary \ref{coro-uni-K0} the inequality $\sum_{i=1}^{r_n} x_i u_{n,i} \leq u_{n^\prime, i^\prime}$ is replaced with equality.

\ 

\paragraph{\bf Acknowledgements.} We would like to thank Ilijas Farah and Eva Perneck\'a for useful conversations and comments.

\bibliographystyle{amsplain}

\end{document}